\documentclass[BCOR=1.5cm, 
twoside=on, 
DIV=11,
twocolumn=false, 
headinclude=false, 
footinclude=false, 
paper=a4, %
pagesize=automedia, 
fontsize=12pt, 
]{scrartcl}

\usepackage[utf8]{inputenc} 
\usepackage[T1]{fontenc} 
\usepackage{amssymb}%
\usepackage{amsmath}%
\usepackage{amsthm}%
\usepackage{mathrsfs} 
\usepackage{enumerate}%
\usepackage{ifthen}
\usepackage{setspace}

\usepackage[a4paper,left=23mm,top=30mm,right=23mm,bottom=40mm]{geometry}

\usepackage[numbers]{natbib}
\usepackage{hyperref}

\let\d\relax
\newcommand{\d}{\,\textup{d}}
\newcommand{\R}{\ensuremath{\mathbb{R}}}%
\let\S\relax
\newcommand{\S}{\ensuremath{\mathbb{S}}}%
\newcommand{\N}{\ensuremath{\mathbb{N}}}%

\let\H\relax
\newcommand{\H}[1][n-1]{\mathcal{H}^{#1}}
\newcommand{\<}{\langle}%
\renewcommand{\>}{\rangle}%
\newcommand{\Ccinfty}[1][(\Omega)]{\ensuremath{C_c^\infty #1}}%
\newcommand{\bd}{\partial}%
\newcommand{\ti}{\tilde}%
\newcommand{\cF}{\mathcal{F}}
\newcommand{\cE}{\mathcal{E}}

\newcommand{\cG}{\mathcal{G}}
\newcommand{\sG}{\mathscr{G}}

\newcommand{\cK}{\mathcal{K}}

\newcommand{\cS}{\mathscr{S}}

\newcommand{\cO}{\mathcal{O}}

\newcommand{\cU}{\mathcal{U}}

\newcommand{\sets}{\;;\;}%
\newcommand{\setsep}{\sets}
\newcommand\sse{\subseteq}
\newcommand\subs\sse
\newcommand{\Per}{\textup{Per}}%

\DeclareMathOperator{\Div}{div}%
\DeclareMathOperator*{\Min}{Min!}

\DeclareMathOperator{\epi}{epi}
\DeclareMathOperator{\Sgn}{Sgn}
\DeclareMathOperator{\sgn}{sgn}
\DeclareMathOperator{\gen}{gen}

\DeclareMathOperator{\Lap}{\Delta}

\DeclareMathOperator*{\esssup}{ess\;sup}
\DeclareMathOperator*{\essinf}{ess\;inf}

\newcommand{\cEP}{\cE_{\Per}}
\newcommand{\cGP}{\cG_{\Per}}

\newtheorem{theo}{Theorem}
\newtheorem{prop}[theo]{Proposition}
\newtheorem{lem}[theo]{Lemma}
\newtheorem{cor}[theo]{Corollary}

\newtheorem{rem}[theo]{Remark}

\numberwithin{equation}{section}
\numberwithin{theo}{section}

\usepackage[automark]{scrpage2}
\ofoot[]{}
\usepackage{varioref}      
\newcounter{nr}  \labelformat{nr}{{\normalfont \alph{nr}}}
\newcommand{\zahl}{{\normalfont (\arabic{nr})}}
\newcommand{\bgl}[1][{\normalfont (\alph{nr})}]
   {\begin{list} {#1} {\usecounter{nr}
          \setlength{\topsep}{0.5ex plus0.2ex minus0.1ex}     
          \setlength{\itemsep}{0.2ex plus0.05ex minus0.03ex}  
          \parsep0pt \itemindent0pt 
          \leftmargin30pt   \labelwidth20pt} }  
\newcommand{\el}{\end{list}}

\title{Perturbation results involving the 1-Laplace operator}
\author{Samuel Littig \and Friedemann Schuricht\footnote{Both authors supported by DFG project “Variational problems
related to the 1-Laplace operator”.}}

\begin{document}
\maketitle


\abstract{We consider perturbed eigenvalue problems of the $1$-Laplace
operator and verify the existence of a sequence of solutions. 
It is shown that the eigenvalues of the perturbed problem converge to the
corresponding eigenvalue of the unperturbed problem as the perturbation
becomes small. The results rely on nonsmooth 
critical point theory based on the weak slope.}

\bigskip


\medskip

\section{Introduction}

Investigations of perturbations of the eigenvalue problem of the $p$-Laplace
operator
\begin{align}
-\Lap_p u +f(x,u)=\lambda |u|^{p-2}u
\quad\mbox{on}\;\Omega\,,\qquad u=0\quad\mbox{on}\;\partial\Omega
\,,\label{e:EVPpLap}
\end{align}
where
\[
\Lap_pu:=\Div |Du|^{p-2}Du \qquad (1<p<\infty)\,, 
\]
gained a lot of interest in the past. A weak solution $u\in
W^{1,p}_0(\Omega)\setminus \{0\}$ is called eigenfunction, the scalar
$\lambda$ eigenvalue, and the tuple $(\lambda, u)$ eigensolution of equation
\eqref{e:EVPpLap}. The function $f$ is considered as perturbation and
one typically assumes that $f$ is small provided $u$ is small, such that
$(\lambda, 0)$ is a trivial eigensolution of \eqref{e:EVPpLap} for any
$\lambda>0$. For $f=0$  we have the (unperturbed) eigenvalue problem of
the $p$-Laplace operator.
 
It is well known that there exists an unbounded sequence of eigenvalues
\begin{equation}   \label{e:int1}
0<\lambda_{1,p}<\lambda_{2,p}\le \lambda_{3,p}\le \ldots
\end{equation}
of the unperturbed $p$-Laplace operator with 
corresponding eigenfunctions $u_{k,p}$. Clearly, any multiple of 
$u_{k,p}$ is also eigenfunction for $\lambda_{k,p}$. Thus
the $(\lambda_{k,p},0)$ are bifurcation points on the branch of
trivial solutions $(\lambda, 0)_{\lambda\in\R}$ of the unperturbed
problem~\eqref{e:EVPpLap} and a natural question is how far this situation is preserved
under small perturbations.

Under suitable assumptions on $f$, the operator $Q: W^{1,p}_0(\Omega)\to
W^{1,p}_0(\Omega)$ with
\[
Q(u):=(-\Lap_p)^{-1}(\lambda |u|^{p-2}u - f(x,u))
\]
turns out to be compact and small as $u$ is small. Hence we may calculate the
Leray-Schauder mapping degree of $u\mapsto u -Q(u)$. 
If the eigenvalue $\lambda_{k,p}$ is simple (which is always the
case for $\lambda_{1,p}$ when $\Omega$ is connected), there exists
a continuous curve $(\lambda_t,u_t)_{t\in\R}$ of eigensolutions of the
perturbed problem \eqref{e:EVPpLap}
crossing the branch of trivial solutions at $(\lambda_{k,p},0)$
(cf.~del~Pino \& Ma\-n\'a\-se\-vich~\cite{delpino-m:91} 
and the survey notes of Peral \cite{peral:97}).
Consequently, if $\lambda_{k,p}$ is simple, 
$(\lambda_{k,p},0)$ is a bifurcation point
of the perturbed problem as well and the eigenvalue $\lambda_{k,p}$ of the
$p$-Laplace operator turns out to be a bifurcation value of the perturbed
$p$-Laplace eigenvalue problem \eqref{e:EVPpLap}. 

A key point in the investiagtion of \eqref{e:EVPpLap} is the underlying
variational structure. In fact the unperturbed problem \eqref{e:EVPpLap}
(i.e. $f=0$) is the
Euler-Lagrange equation of the variational problem 
\begin{align}
\cE_p(v):=\frac{1}{p}\int_\Omega |Dv|^p\d x \;\to \Min_{v\in
  W^{1,p}_0(\Omega)}\label{e:Ep} 
\end{align}
subject to
\begin{align}
\cG_p(v):=\frac{1}{p}\int_\Omega |v|^p\d x=1\,.\label{e:Gp}
\end{align}
In other words, any critical point $u$ of \eqref{e:Ep}, \eqref{e:Gp} turns out
to be an eigenfunction of the $p$-Laplace operator for the eigenvalue
$\lambda=p\,\cE_p(u)$ (which equals the Lagrange multiplier of the constrained
variational problem). Moreover each eigenfunction of the unperturbed equation 
\eqref{e:EVPpLap} is a multiple of a
critical point of \eqref{e:Ep}, \eqref{e:Gp}. 

Notice that an unbounded sequence of eigenvalues $\lambda_{k,p}$ of the 
$p$-Laplace operator as mentioned in \eqref{e:int1} can be obtained by minimax
methods within Ljus\-ter\-nik-Schni\-rel\-man
 theory where one has
\begin{equation} \label{e:int2}
\lambda_{k,p}=\inf_{S\in \cS^{k,p}}\sup_{v\in S}\:\cE_p(v)\,.
\end{equation}
Here the $\cS^{k,p}$ are suitable classes of subsets of $W^{1,p}_0(\Omega)$ 
expressing some topological property of the level sets of $\cE_p$ by means of
some topological index $k$. 
It is well known that these eigenvalues are continuous in $p$ on $[1,\infty)$
(cf.~Parini~\cite{parini:11}, Littig \& Schuricht \cite{littig-s:13}).
For $p>1$ the eigenvalue problem is studied in wide detail with contributions
of many authors. Let us just mention Garcia Azorero \& Peral Alonso
\cite{azorero-a:87}, who seem to have studied the problem first, and the long
list of references contained in Peral~\cite{peral:97}. 

When studying \eqref{e:EVPpLap}, one usually distinguishes three cases
depending on the growth of $f$. Here a typical assumption on $f$ is 
\[
|f(x,u)|\le C(1 +|u|^{r-1}) \,.
\]
For $1\le r< p^*:=\frac{np}{n-p}$ the problem is called subcritical, for
$r=p^*$ it is called critical, and for $r>p^*$ it is called
supercritical. If $p\ge n$ the problem is always subcritical. Usually the
subcritical case is the most easy one to treat. In the critical case we may
expect similar results as in the subcritical case, but the techniques for the
proofs are more involved. In the supercritical case nonexistence of solutions
may occur (cf.~\cite{peral:97}). 

The intention of the present paper is to study such bifurcation problems for
the degenerate limit case $p=1$. Taking into account two types of
perturbations, we cover problems that are formally given by 

\begin{align}\label{e:int3}
-\Div \frac{Du}{|Du|} + f(x,u) =\lambda \frac{u}{|u|}\,
\quad\mbox{on}\;\Omega\,,\qquad u=0\quad\mbox{on}\;\partial\Omega\,,
\end{align}
and by
\begin{align}\label{e:int4}
-\Div \frac{Du}{|Du|} = \lambda \Big(\frac{u}{|u|}+g(x,u)\Big)
\quad\mbox{on}\;\Omega\,,\qquad u=0\quad\mbox{on}\;\partial\Omega
\,.
\end{align}

Notice that already the (unperturbed) eigenvalue problem of the $1$-Laplace
operator (i.e. $f=0$ or $g=0$) is highly degenerated, since the equations
above are not well defined at points where $u(x)=0$ or $Du(x)=0$. 
Having in mind that typically the first eigenfunction of the $1$-Laplace operator 
is a multiple of a characteristic function vanishing on a set of positive
measure, it becomes clear that the equations need some careful justification.  
Instead of $W^{1,1}_0(\Omega)$ one has to work in $BV(\Omega)$ and the
homogeneous boundary conditions have to be considered in a more general sense
than the usual trace in $BV(\Omega)$. Then it turns out that the unperturbed problem
is related to the variational problem 
\begin{equation}
\cE_{TV}(v):=\int_\Omega \d |Dv| + \int_{\bd\Omega}|v|\d
\H\;\to \Min_{v\in BV(\Omega)}\label{e:etv}
\end{equation}\pagebreak[0]
subject to\pagebreak[0]
\begin{align}
\cG_1(v)=\int_\Omega |v|\d x=1\,\label{e:g1}
\end{align}
(cf.~Kawohl \& Schuricht \cite{kawohl-s:07}).
With methods from convex analysis and nonsmooth critical point theory one can
show that critical points of problem \eqref{e:etv}, \eqref{e:g1} (in the sense
of weak slope) satisfy the   
Euler-Lagrange equation 
\begin{align}\label{e:int5}
-\Div z = \lambda s \qquad\mbox{on }\Omega\,.
\end{align}
Here $z\in L^\infty(\Omega)$ is some vector field giving sense to 
$\frac{Du}{|Du|}$ and $s:\Omega\to[-1,1]$ is some sign function
giving sense to $\frac{u}{|u|}$ (cf.~\cite{kawohl-s:07}).
Existence of a sequence of eigenfunctions
$(u_{k,1})_{k\in\N}$ of the $1$-Laplace operator with an unbounded sequence of
corresponding eigenvalues 
\begin{align}
\lambda_{k,1}=\inf_{S\in\cS^\alpha_k}\sup_{v\in S}\:\cE_{TV}(v)\,\label{e:la1k1Lap}
\end{align}
was verified in Milbers \& Schuricht \cite{milbers-s:10a} by minimax methods. 
While in \cite{milbers-s:10a} the classes $\cS^\alpha_k$ are defined by means
of category as topological index, 
we know from Littig \& Schuricht \cite{littig-s:13} that these eigenvalues
$\lambda_{k,1}$ coincide with that using 
\[
\cS^\alpha_k:=\{ S\subseteq L^1(\Omega) \text{ compact, symmetric}\setsep
\cG_1=1\text{ on }S,\:\gen_{L^1} S \ge k\}\, 
\]
with genus $\gen_{L^1} S$ as topological index in \eqref{e:la1k1Lap}.

Investigating bifurcation for the formal problems \eqref{e:int3} and
\eqref{e:int4} we are confronted with the question how to define solutions.
We have to realize that even in the unperturbed case the well-defined
interpretation \eqref{e:int5} of the formal equation has too many solutions
and cannot identify reasonable solutions of the problem
(cf.~Kawohl \& Schuricht \cite{kawohl-s:07}, Milbers \& Schuricht \cite{milbers-s:12}). Therefore we have to define
solutions of \eqref{e:int3} and \eqref{e:int4} as critical points (in the
sense of weak slope) of a related variational problem. 
In this sense we first verify the existence of a
sequence of eigensolutions $(\lambda_{k,\alpha},u_{k,\alpha})_{k\in\N}$ with
critical values $c_{k,\alpha}$ for a class of problems covering 
\eqref{e:int3} and a sequence of
eigensolutions $(\lambda_{k,\beta},u_{k,\beta})_{k\in\N}$ with critical values
$c_{k,\beta}$ for a class of problems covering \eqref{e:int4} 
for each sufficiently small parameter $\alpha>0$ and $\beta>0$, respectively.
In both cases we assume that the perturbation is of subcritical type,
i.e.~$1<p<\frac{n}{n-1}$\,. 
The parameters $\alpha$ and $\beta$ correspond to the norm of the
eigenfunctions and, thus, they reflect the magnitude
of the perturbation. The perturbation is shown to vanish as
$\alpha$ or $\beta$ tend to zero provided we have the 
stronger condition $1<p< \frac{n+1}{n}$.

Finally we prove 
\begin{align}\label{e:int6}
\lambda_{k,1}=\lim_{\alpha\to 0}\lambda_{k,\alpha}\qquad \text{and}\qquad
\lambda_{k,1}=\lim_{\beta\to 0}\lambda_{k,\beta}
\end{align}
for any $k\in\N$. Since all points $(\lambda,0)_{\lambda\in\R}$ may be
considered as 
trivial solution of the perturbed eigenvalue problem, \eqref{e:int6} shows
that the minimax eigenvalues $\lambda_{k,1}$ of the (unperturbed) $1$-Laplace
operator according to \eqref{e:la1k1Lap} are bifurcation values 
for the perturbed eigenvalue problems \eqref{e:int3} and \eqref{e:int4}.

Let us mention that De\-giovan\-ni~\&~Magro\-ne~\cite{degiovanni-m:09}
treated the critical case (which is not covered by our results). 
They proved existence of
nontrivial solutions of \eqref{e:int3} for any $\lambda>\lambda_{k,1}$ and the
perturbation $f(x,u)=-|u|^{\frac{n}{n-1}-2}u$. Here
a different method is used that relies on truncation techniques of
$BV$-functions and exploits the specific form of the perturbation. 

In Sec\-tion~\ref{s:formulation}
we precisely formulate the two types of perturbed eigenvalue
problems and justify related quantities. The main results are stated in
Sec\-tion~\ref{s:results}. As preparation for the proofs, Sec\-tion~\ref{s:tools}
presents tools from nonsmooth critical point theory and some general norm
estimates. In Sec\-tion~\ref{s:proofs} we give the proofs of the main results. 

\medskip

\textbf{Notation and Conventions:} 
By $L^p(\Omega)$ we denote the usual Lebesgue space of $p$-integrable functions 
with norm $\|\cdot\|_p$ and by
$W^{1,p}_0(\Omega)$ the Sobolev space of $p$-integrable functions having
$p$-integrable weak derivatives and zero trace.
$C_{\textup{BV}}$ is the embedding constant of
$W^{1,1}_0(\Omega)$ (with norm $\|Dv\|_1$) in
$L^{\frac{n}{n-1}}(\Omega)$, i.e. it is the optimal constant in
\begin{equation} \label{e:int8}
\|v\|_{\frac{n}{n-1}} \le  C_{\textup{BV}} \|Dv\|_1   \quad\text{for all } v\in
W^{1,1}_0(\Omega)\,.
\end{equation}
$BV(\Omega)$ stands for the space of functions of bounded variation.
With the usual convention of identifying $v\in L^1(\Omega)$ with its
extension by zero on $\R^n\setminus \Omega$, we have for all 
$v\in BV(\Omega)$  
\begin{align}
\cE_{TV}(v)=\int_{\R^n} \d |Dv| = \int_\Omega \d |Dv| + \int_{\bd\Omega}|v|\d
\H\label{e:defETV} 
\end{align}
(cf.~\cite{evans-g:92}). Due to Theorem~3.1 in 
\cite{littig-s:13} and the Poincaré inequality,
$\cE_{TV}$ is a norm on $BV(\Omega)$ equivalent to the standard norm. 
$\Ccinfty$ is the space of test functions having compact support.

We write $\lambda_{k,p}$ ($p> 1$) for the variational eigenvalues
of the $p$-Laplace operator as given in \eqref{e:int2} and
$\lambda_{k,1}$ always stands for the eigenvalues of the (unperturbed)
$1$-Laplace operator according to \eqref{e:la1k1Lap}. Without danger of
confusion we use $\lambda_{k,\alpha}$ and $\lambda_{k,\beta}$ for the 
eigenvalues of the perturbed $1$-Laplace operator for perturbations
of the type given in Section \ref{ss:pertenergy} and Section
\ref{ss:pertconst}, respectively. Analogously we denote the corresponding
eigenfunctions and critical values.

The set-valued sign function $\Sgn$ on $\R$ is  
\begin{equation}
\Sgn(s) := \left\{
\begin{array}{ccl} \{1\} && \text{if } s>0\,,\\
                  \mbox{$[-1,1]$} && \text{if } s=0\,,\\
                  \{-1\} && \text{if } s<0\,
          \end{array}
           \right.   
\end{equation}
and $\S^{k-1}$ denotes the $(k-1)$-dimensional unit sphere in $\R^k$.

For a Banach space $X$ and its dual $X^*$, the duality pairing is given by 
$\<\cdot,\cdot\>$. By $B_\rho(u)$ we denote the open $\rho$-ball around $u$, 
by $B_\rho(M)$ the open $\rho$-neighborhood of the set $M$, by 
$\overline M$ the closure of $M$, by $I_M$ the indicator function of $M$,
and by $\chi_M$ the characteristic function of $M$.
We write $\gen_X S$ for the genus of a symmetric set $S\subs X\setminus\{0\}$  
(cf.~\cite[Chap.~44.3]{zeidler:84} for basic properties). 
For a scalar function $\cF:X\to\R$ we use $\partial \cF(u)$ to denote 
the convex subdifferential at $u$ for a convex $\cF$ and 
Clarke's generalized gradient at $u$ for a locally Lipschitz
continuous $\cF$. Clarke's generalized directional derivative of $\cF$ 
at $u$ in direction $v$ is given by $\cF^0(u;v)$ 
(cf.~\cite{clarke:87}). 
For a continuous or merely 
lower semicontinuous $\cF:M\to\R$ on a metric space $M$, 
the weak slope of $\cF$ at $u$, denoted by
$|d\cF|(u)$, is a nonnegative real number that describes somehow the slope of
$\cF$ on some neighborhood of $u$ and can be considered as some replacement of
$\|\cF'(u)\|$ in the smooth case 
(cf.~Section~\ref{ss:nonsmooth} for some brief introduction).

\section{Formulation of the problem}
\label{s:formulation}

We always assume that $\Omega\subseteq\R^n$ is open and bounded with Lipschitz
boundary and that $1<p<\infty$. 
First we study perturbed eigenvalue problems that cover problems 
formally given by
\begin{align}
-\Div \frac{Du}{|Du|}+f(x,u)=\lambda \frac{u}{|u|}
\quad\text{ on }\Omega\,, \quad u=0 \quad\text{ on }\partial \Omega\,. 
\label{e:formEVP}
\end{align}
More precisely we consider 
critical points of a constraint variational problem
\begin{align}
\cE_{TV}(v)+\cE_{\Per}(v)\to \Min_{v\in L^p(\Omega)} \label{e:E_per1Lap}
\end{align} subject to \begin{align}
\cG_1(v):=\int_\Omega|v|\d x = \alpha\,.\label{e:cG_Eper}
\end{align}
Here $\cEP:L^p(\Omega)\to\R$ be a suitable locally
Lipschitz continuous functional and
we identify $\cE_{TV}$ with its extension on $L^p(\Omega)$ 
for $1< p< \tfrac{n}{n-1}$ given by
\begin{align}\label{e:form4}
  \cE_{TV}(v):= \left\{  
         \begin{array}{ll} \int_\Omega \d |Dv| + \int_{\bd\Omega}|v|\d\H \quad & 
                              \text{on } BV(\Omega)\,,\\
                              \infty & \text{on } L^p(\Omega)\setminus
                              BV(\Omega)\,.
          \end{array}                     \right. 
\end{align}
Obviously $\cE_{TV}$ is convex and it is the 
the lower semicontinuous extension 
of $\int_\Omega |Dv|\d x$ from $v\in W^{1,1}_0(\Omega)$ on 
$L^p(\Omega)$ (cf.~\cite{kawohl-s:07} and \cite{littig-s:13}).
We call $u\in L^p(\Omega)$ critical point 
of \eqref{e:E_per1Lap}, \eqref{e:cG_Eper} if $u$ is a critical point 
with respect to the weak slope of 
\[  \cE = \cE_{TV}+\cE_{\Per}    \]
in the metric space
\[ K_\alpha:=\{v\in L^p(\Omega)\setsep \cG_1(v)=\alpha\}\,, \]
i.e.~if $|d\cE|(u)=0$ (cf.~Section~\ref{ss:nonsmooth} and 
Degiovanni \& Marzocchi~\cite{degiovanni-m:94}).
This is equi\-va\-lent to $u$ being a critical point of 
\begin{align}
\tilde\cE = \cE_{TV}+\cE_{\Per} + I_{\{\cG_1=\alpha\}}\,\label{e:EperCritFunc}
\end{align}
on the metric space $L^p(\Omega)$, i.e. $|d\tilde\cE|=0$, where 
\[
I_{\{\cG_1=\alpha\}}(v)=\begin{cases}0&\text{when } \cG_1(v)=\alpha\\ \infty
  &\text{otherwise}\end{cases} 
\]
is the indicator function of $K_\alpha$ (cf.~also Milbers \& Schuricht 
\cite{milbers-s:10}).

With this definition at hand we can apply a nonsmooth version of 
Ljus\-ter\-nik-Schni\-rel\-man theory to get a sequence $(u_{k,\alpha})_{k\in\N}$ of
eigenfunctions of the perturbed problem \eqref{e:E_per1Lap}, \eqref{e:cG_Eper}
for each parameter $\alpha>0$. As necessary condition each eigensolution
$(\lambda,u)$ satisfies an Euler-Lagrange equation of the type 
\[
-\Div z + u^* = \lambda s\,
\]
where $z$ and $s$ are related to $u$ as in the unperturbed case 
(cf.~\eqref{e:int5}) and $u^*\in\partial\cE_{Per}(u)$.
Also in the perturbed case the parameter $\lambda\in\R$ will be called
eigenvalue of the eigenfunction $u$. Essential ingredients in our analysis 
will be some Pa\-lais-Smale condition (short (PS)-con\-di\-tion), which requires
special care, and the so-called epigraph condition (short (epi)-con\-di\-tion)
that rules out ``artificial'' critical points on the epigraph of our merely
lower semicontinuous functional and it can be treated rather 
straightforward.

As a second type of perturbation we cover problems formally given by
\begin{align}
\label{e:form5}
-\Div \frac{Du}{|Du|}=\lambda \left(\frac{u}{|u|}+g(x,u)\right)
\quad\text{ on }\Omega\,, \quad u=0 \quad\text{ on }\partial \Omega\,. 
\end{align}
More precisely we study critical points of constrained variational
problems 
\begin{align}
\cE_{TV}(v)\to \Min_{v\in L^p(\Omega)} \label{e:form6}
\end{align} subject to \begin{align}
\cG_1(v)+\cG_{Per}(v) = \beta\,\label{e:form7}
\end{align}
where $\cGP:L^p(\Omega)\to\R$ is a suitable locally Lipschitz continuous
functional. Here $u\in L^p(\Omega)$ is a critical point of \eqref{e:form6},
\eqref{e:form7} if $u$ is a critical point of $\cE_{TV}$ in the metric space
\[ K_\beta:=\{v\in L^p(\Omega)\setsep \cG_1(v)+\cGP(v)=\beta\}\,, \]
i.e. $|d\cE_{TV}|(u)=0$ or, equivalently, if $u$ is a critical point of
\[
\cE=\cE_{TV} + I_{\{\cG_1 + \cG_{Per}=\beta\}}
\]
on $L^p(\Omega)$, i.e. $|d\cE|(u)=0$.
This way we again obtain  
a sequence $(u_{k,\beta})_{k\in\N}$ of eigenfunctions of the perturbed problem 
\eqref{e:form6}, \eqref{e:form7}
for each parameter $\beta>0$ and the eigensolutions  
satisfy an Euler-Lagrange equation of the form 
\[
-\Div z = \lambda (s +u^*)
\]
where $z$ and $s$ are related to $u$ as before and $u^*\in \partial
\cG_{Per}(u)$. Again $u$
will be called eigenfunction and $\lambda$ the corresponding eigenvalue for
this type of perturbation. In contrast to the perturbation of the first type,
the (PS)-con\-di\-tion is a simple immediate consequence of the compact embedding
from $BV(\Omega)$ in $L^p(\Omega)$, but the verification of the
(epi)-con\-di\-tion turns out to be more delicate.

\begin{rem}
Note that the weak slope and, thus, our definition of criticality
  depends on the specific choice 
  of the metric. This issue was already addressed in Milbers \& Schuricht \cite{milbers-s:10} and Littig \& Schuricht \cite{littig-s:13}. It turns out that, for any $p\in [1,\frac{n}{n-1})$,
  the minimax construction as in \eqref{e:la1k1Lap} 
  provides eigenfunctions $u_{k,1}$ of the (unperturbed) $1$-Laplace operator
  that are critical points in $L^p(\Omega)$ with respect to the $L^p$-metric.   
  However it is not clear whether $L^p$-critical points are 
  also $L^q$-critical points for $p\neq q$ in general. 
  Alternatively one could consider critical points in $BV(\Omega)$ with
  respect to the stronger $BV$-norm. In the one-dimensional case, for $\Omega=(0,1)\subseteq\R$, it can be
  shown that this leads to a much larger set of critical points 
  and it seems that, in general, the $BV$-norm is too strong 
  to get a reasonable set of critical points (cf.~Milbers \& Schuricht \cite{milbers-s:12}). 
\end{rem}

As prototype for perturbations $\cEP$, $\cGP$ we have in mind
functionals of the form 
\begin{align}\label{e:GPerDefn}
\cE_{\Per}(v)=\int_{\Omega}\int_0^{v(x)}f(x,s)\d s \d x\,,
\qquad
\cG_{\Per}(v)=\int_{\Omega}\int_0^{v(x)}g(x,s)\d s \d x\,.
\end{align}
In order to derive general properties for this kind of functionals we
use the notation
\begin{align}
\cF(v):=\int_{\Omega}\int_0^{v(x)}f(x,s)\d s \d x\,.\label{e:form8}
\end{align} 
For the integrand $f:\Omega\times\R\to\R$ we assume that 
\bgl
\item[(f1)] \label{e:form1}
$f$ is locally integrable on $\Omega\times\R$ and 
$F:\Omega\times \R\to \R$ defined by
\[ F(x,t):=\int_{0}^tf(x,s)\d s \]
is a Carathéodory function,

\item[(f2)] \label{e:form2}
there is $C_{\Per}>0$ such that
\begin{align}\label{e:fleCsp-1}
|f(x,s)|\le p\, C_{\Per}\,|s|^{p-1} \quad
\text{for all } s\in\R \text{\; and a.e. }x\in\Omega \,,
\end{align}

\item[(f3)] \label{e:form3}
$f(x,\cdot)$ is odd for a.e. $x$, i.e. 
\begin{align} \label{e:antsymf}
f(x,s)=-f(x,-s) \quad \text{for all } s\in\R \text{\; and a.e. }x\in\Omega \,,
\end{align}
\el
A standard example for $f$ would be
\[
f(x,s)= |s|^{p-2}s  \,.   
\]
In the following theorem, which is proved in Section \ref{ss:proofcF}, we 
summarize several properties of functional $\cF$ given in \eqref{e:form8}.
\begin{theo}\label{t:cF}
Let $\Omega\subseteq\R^n$ be open and bounded with Lipschitz boundary
and let $f$ satisfy conditions \upshape{(f1)-(f3)} with
$p\in (1,\infty)$. Then 
$\cF: L^p(\Omega)\to \R$ according to \eqref{e:form8} is well defined with
\bgl[\zahl]
\item
$\cF$ is even, i.e. $\cF(u)=\cF(-u)$ for all $u\in L^p(\Omega)$,
\item
$\cF(u) = \cF(|u|)$ for all $u\in L^p(\Omega)$,
\item
$\cF$ is Lipschitz continuous on bounded subsets of $L^p(\Omega)$ and, thus,
locally Lipschitz continuous on $L^p(\Omega)$,
\item for $u^*\in\partial\cF(u)\subseteq L^{p'}(\Omega)$ one has  
\begin{equation}
u^*(x)\in [\essinf_{s\to u(x)}f(x,s), \esssup_{s\to u(x)}f(x,s)]
\qquad\text{for a.e. } x\in \Omega \,,\label{e:u*f-f+}
\end{equation}
\begin{equation}
 \|u^*\|_{p'}\le p\,C_{\Per}\,\|u\|_p^{p-1}\,,\label{e:|u*|p'}
\end{equation}

\item and for all $u\in L^p(\Omega)$ and all $\alpha\ge0$
\begin{align}
 |\cF(\alpha u)|\le \alpha^p\, C_{\Per}\|u\|_p^p\,.\label{e:cFleC|u|_p^p}
\end{align}
\el
\end{theo}

\begin{rem}
\begin{enumerate}
\renewcommand\labelenumi{\textup{(\theenumi)}}
\item
If $f(x,\cdot)$ is continuous for a.e. $x\in\Omega$ , by \eqref{e:u*f-f+}
Clarke's generalized gradient  
$\partial\cF(u)$ reduces to a singleton $\partial\cF(u)=\{u^*\}$ with 
\[
u^*(x)=f(x,u(x)) \qquad \text{for a.e. }x\in\Omega\,.
\]
In this sense the generalized gradient $\partial\cF(u)$ extends 
the classical Nemytskii operator
$u\mapsto f(\cdot, u(\cdot))$ as formally used in \eqref{e:formEVP} and
\eqref{e:form5}. 

\item
Let us mention that the theorem remains true for $\Omega\subseteq\R^n$
merely open.

\item
With perturbations $\cEP$ and $\cGP$ of the form \eqref{e:GPerDefn},
our main results stated in Section \ref{s:results} do not need further
conditions than \textup{(f1)-(f3)} for the integrand $f$ of $\cEP$, but our verification of the (epi)-con\-di\-tion requires a slightly stronger
assumption on the integrand $g$ of $\cGP$ (cf.~\eqref{e:u^*u>-|u|} below).

\end{enumerate}
\end{rem}

\section{Main results}
\label{s:results}

In this section we state the main results while the essential proofs
are postponed to Section~\ref{s:proofs} and some preliminary results are
presented in Section~\ref{s:tools}. We always assume that 
$\Omega\subseteq\R^n$ is open and bounded with Lipschitz boundary.

\subsection{Perturbation of the energy}
\label{ss:pertenergy}

For $\alpha>0$ and $1<p<\frac{n}{n-1}$ we investigate the
perturbed eigenvalue problem of the $1$-Laplace operator 
\begin{equation}
\cE_{TV}(v)+\cE_{\Per}(v)\to \Min_{v\in L^p(\Omega)} \label{e:E_per1Lap-a}
\end{equation}
\begin{equation}
\cG_1(v) = \alpha\,.\label{e:cG_Eper-a}
\end{equation}
(cf.~\eqref{e:E_per1Lap}, \eqref{e:cG_Eper}).\pagebreak[4]

Recall that we defined
eigenfunctions to be critical points of 
\[ \cE_{TV}+\cE_{\Per} + I_{\{\cG_1=\alpha\}} \] 
on $L^p(\Omega)$. 
For the perturbation function $\cE_{\Per}: L^p(\Omega)\to \R$ we assume that
\bgl
\item[(E1)] $\cE_{\Per}$ is locally Lipschitz continuous on $L^p(\Omega)$,
\item[(E2)] $\cE_{\Per}$ is even, i.e. $\cE_{\Per}(v)=\cE_{\Per}(-v)$
for all $v\in L^p(\Omega)$,
\item[(E3)] there is a constant $C_{\Per}>0$ such that for all $v\in
  L^p(\Omega)$ 
\begin{align}
|\cE_{\Per}(v)|\le C_{\Per}\|v\|_p^p  \,,\label{e:Eper_le_||p}
\end{align}
\item[(E4)] for all $v^*\in\partial \cE_{\Per}(v)$ and all
$v\in L^p(\Omega)$ one has
\begin{align}
\|v^*\|_{p'}\le p\,C_{\Per}\|v\|_p^{p-1}\,.\label{e:u*leu_p^p}
\end{align}
\el
Notice that all these conditions are fulfilled in the case where 
\begin{align} \nonumber
\cE_{\Per}(v)=\int_{\Omega}\int_0^{v(x)}f(x,s)\d s \d x\label{e:GPerDefn}
\end{align}
and the integrand $f$ satisfies (f1)-(f3) (cf.~Theorem~\ref{t:cF} above).

Since $\cE_{TV}$ is lower semicontinuous on $L^p(\Omega)$, the functional
$\cE_{TV}+\cE_{\Per}$ turns out to be lower semicontinuous on
$L^p(\Omega)$ too. By definition, $u\in L^p(\Omega)$ is an eigenfunction of our
perturbed $1$-Laplace problem if it is a critical point of
\eqref{e:E_per1Lap-a},\,\eqref{e:cG_Eper-a} in the sense of the weak slope,
i.\,e. $|d\cE|(u)=0$ for 
\[\cE:=\cE_{TV}+\cE_{\Per} + I_{\{\cG_1=\alpha\}}\]
with $\alpha=\cG_1(u)$. Let us first formulate some Euler-Lagrange equation as
necessary condition for critical points of that problem. The proof can be
found in Section \ref{ss:proofpertenergy} below.

\begin{theo}[Eu\-ler-La\-grange Equation]\label{t:ELG}
Let $\Omega\subseteq\R^n$ be open and bounded with Lipschitz boundary, let
$1<p<\frac{n}{n-1}$, let $\cEP$ satisfy \textup{(E1)-(E4)}, and let  
$u$ be a critical point of variational problem
\eqref{e:E_per1Lap-a},\,\eqref{e:cG_Eper-a} for some $\alpha>0$. 
Then there exists a function $s\in L^{\infty}(\Omega)$ with 
\[
s(x)\in \Sgn(u(x)) \qquad\text{for a.e. }x\in \Omega\,,
\]
a vector field $z\in L^{\infty}(\Omega,\R^n)$ with
\[
\Div z\in L^{p'}(\Omega)\,,\quad\|z\|_\infty= 1\,,
\quad \text{ and }\quad\cE_{TV}(u) =- \int_\Omega u  \Div z\d x\,,
\]
some $u^*\in \partial \cE_{\Per}(u)\subseteq L^{p'}(\Omega)$
and a Lagrange multiplier $\lambda\in\R$ such
that the Eu\-ler-La\-grange equation 
\begin{align}
 -\Div z + u^* = \lambda s \qquad\text{on }\:\Omega    \label{e:ELGEper}
\end{align}
is satisfied.

In the case where 
\begin{align} \label{e:results1}
\cE_{\Per}(v)=\int_{\Omega}\int_0^{v(x)}f(x,s)\d s \d x 
\end{align}
with $f$ satisfying \textup{(f1)-(f3)} we have  
\[ 
u^*(x)\in  \Big[\essinf_{s\to u(x)}f(x,s), \esssup_{s\to u(x)}f(x,s)\Big]\quad
\text{ for a.e. $x\in\Omega$\,.} 
\]
If, in addition, $f(x,\cdot)$ is continuous for a.e. $x\in\Omega$, then 
$u^*(x)=f(x,u(x))$ a.e. on $\Omega$ and \eqref{e:ELGEper} becomes
\begin{align}\label{e:results1a}
 -\Div z + f(x,u) = \lambda s \qquad\text{on }\:\Omega\,.
\end{align}
\end{theo}

\begin{rem}\label{r:results1a}
\begin{enumerate}
\renewcommand\labelenumi{\textup{(\theenumi)}}
\item
In contrast to the differentiable case of the $p$-Laplace operator with $p>1$,
we cannot expect that the contrary of Theorem \ref{t:ELG} is true,
since a function $u$ satisfying the Eu\-ler-La\-grange
equation \eqref{e:ELGEper} doesn't need to be a a critical point of
\eqref{e:E_per1Lap-a},\,\eqref{e:cG_Eper-a}. This fact is already known for
the unperturbed case $f=0$ (cf.~\cite{milbers-s:12}).

\item 
Using the eigenfunction $u$ as a test function in \eqref{e:ELGEper},
we obtain for the corresponding eigenvalue
\begin{align}   \label{e:results5}
\lambda=\frac{\cE_{TV}(u) + \<u^*,u\>}{\alpha}\,
\end{align}
for some $u^*\in\partial\cEP(u)$.
In the unperturbed case where $f=0$ and thus $u^*=0$, we have 
$\lambda= \frac{\cE_{TV}(u)}{\alpha}$. Hence the eigenvalue $\lambda$ is
uniquely determined by the eigenfunction $u$, though the functions $(z,s)$ in 
\eqref{e:ELGEper} related to $u$ might be not unique (for the first
eigenfunction of the $1$-Laplace operator we even know that $(z,s)$ are not
unique in general, cf.~Kawohl \& Schuricht \cite{kawohl-s:07}).

In the general perturbed situation it is not clear 
if the eigenvalue $\lambda$ associated to an eigenfunction $u$ is uniquely
determined. It might happen that there are solutions   
$(s_1,z_1,u^*_1,\lambda_1)$ and $(s_2,z_2,u^*_2,\lambda_2)$ of
\eqref{e:ELGEper}, both related to eigenfunction $u$ but with different 
eigenvalues $\lambda_1\neq \lambda_2$. However, 
if $\cEP$ has the form \eqref{e:results1}, this
can only occur in the irregular case when for all $x$ from a set $E$ of positive measure the function $f(x,\cdot)$ is not continuous.
\end{enumerate}
\end{rem}

We know that the (unperturbed) eigenvalue problem for the $1$-Laplace operator
has a sequence of eigensolutions $(\lambda_{k,1},u_{k,1})$ that can be
constructed by methods of Ljus\-ter\-nik-Schni\-rel\-man
 theory
(cf.~\cite{milbers-s:10}). Since the underlying
minimax principle has some robustness against perturbations, we now want to
show that the perturbed eigenvalue problem
\eqref{e:E_per1Lap-a},\eqref{e:cG_Eper-a}  
has a sequence of eigensolutions $(\lambda_{k,\alpha},u_{k,\alpha})$ for each
$\alpha>0$ sufficiently small.

In critical point theory the Palais-Smale or (PS)-con\-di\-tion ensures some
compactness. In our nonsmooth context the lower semicontinuous function  
$\cE:M\to\R\cup\{\infty\}$ on the metric space $M$ is said to satisfy the 
(PS)-con\-di\-tion at level $c\in\R$ if any Palais-Smale sequence $(u_j)_j$, 
i.e. $\cE(u_j)\to c$ and $|d \cE|(u_j)\to 0$, admits a convergent subsequence.
If $\cE$ satisfies the (PS)-con\-di\-tion at all levels $c\in\R$, we simply say
that $\cE$ satisfies the (PS)-con\-di\-tion. 

\begin{prop}[$(\textup{PS})$-condition]\label{p:PS-Eper}
Let $\Omega\subseteq\R^n$ be open and bounded with Lipschitz boundary, let
$1<p<\frac{n}{n-1}$, let $\cEP$ satisfy \textup{(E1)-(E4)}, and let $\alpha>0$.
Moreover we assume that one of the following conditions holds:
\bgl
\item[\textup{(E5')}] $\cE_{\Per}$ is globally bounded from below on $L^p(\Omega)$
  or 

\item[\textup{(E5'')}]\label{p:PS_b} 
with the embedding constant $C_{\textup{BV}}$ of $W^{1,1}_0(\Omega)$ in
$L^{\frac{n}{n-1}}(\Omega)$ (cf.~\eqref{e:int8}) we have 
\[ p\le 1+\frac{1}{n} \]
and 
\begin{align}
  \alpha^{n-(n-1)p}< C_{\Per}^{-1}\, C_{\textup{BV}}^{-(p-1)n}\,.\label{e:alph<<1}
\end{align}
\el
Then the \textup{(PS)}-condition is satisfied for
$\cE=\cE_{TV}+\cE_{\Per}+I_{\{\cG_1=\alpha\}}$ on $L^p(\Omega)$. 
 \end{prop}

\begin{rem}
Condition \textup{(E5')} is trivially satisfied provided $\cE_{\Per}$ is of the
form \eqref{e:results1} and the integrand $f$ is bounded from below on 
$\Omega\times [0,\infty)$.
Assumption $p\le 1 + \frac{1}{n}$, or equivalently
$n-(n-1)p\ge\frac{1}{n}$, implies that \eqref{e:alph<<1} can always
be achieved for $\alpha$ sufficiently close to zero. 
\end{rem}

We use the genus as topological index for the minimax construction of critical
points. As genus $\gen_X S$ of a symmetric $S\subseteq X\setminus \{0\}$
in a Banach space $X$ we define the
least integer $k\in\N$ such that there exists an odd continuous map 
$\phi: S\to \R^k\setminus \{0\}$ and we set $\gen_X S=\infty$ provided such a
map doesn't exist at all (cf.~\cite[Chap.~44.3]{zeidler:84}). 

\begin{theo}[Existence of eigensolutions]\label{t:exist:CP1perLap}
Let $\Omega\subseteq\R^n$ be open and bounded with Lipschitz boundary, let
$1<p<\frac{n}{n-1}$, let $\alpha>0$, and 
let $\cEP$ satisfy \textup{(E1)-(E4)} and either \textup{(E5')} or \textup{(E5'')}.
Then there exists a sequence of eigenfunctions 
$(\pm u_{k,\alpha})_k$ of the perturbed
eigenvalue problem \eqref{e:E_per1Lap-a},\,\eqref{e:cG_Eper-a}
with $\cG_1(\pm u_{k,\alpha})=\alpha$ and 
where the corresponding critical values 
$c_{k,\alpha}=\cE_{TV}(\pm u_{k,\alpha})+\cE_{\Per}(\pm u_{k,\alpha})$ 
are characterized by  
\begin{align}
c_{k,\alpha}=\inf_{S\in\cS^\alpha_k}\sup_{v\in S}\;
\cE_{TV}(v)+\cE_{\Per}(v)\,\label{e:hatck} 
\end{align}
with 
\begin{align}
\cS^\alpha_k:=\{ S\subseteq L^p(\Omega) \text{ compact, symmetric}\setsep
\cG_1=\alpha\text{ on }S,\: \gen_{L^p} S \ge k\}\,. \label{e:cSkalpha} 
\end{align}
The sequence of critical values $(c_{k,\alpha})_k$ is unbounded.
For each $k\in\N$ there is some $\alpha_1>0$ such that
the family of rescaled eigenfunctions 
$v_{k,\alpha}:=\frac{u_{k,\alpha}}{\|u_{k,\alpha}\|_1}$ is bounded in
$BV(\Omega)$ for $0<\alpha\le\alpha_1$ 
 and any choice of eigenfunctions $u_{k,\alpha}$ having critical value 
$c_{k,\alpha}$.
Moreover, the Eu\-ler-La\-grange equation \eqref{e:ELGEper} from Theorem
\ref{t:ELG} holds for any critical point $u_{k,\alpha}$. 
\end{theo}

The proof of Theo\-rem~\ref{t:exist:CP1perLap} is given in
Section~\ref{ss:proofpertenergy} and essentially relies on a general
existence result for critical points of lower semicontinuous functional 
stated in Theo\-rem~\ref{t:exCP} in Section~\ref{ss:nonsmooth}. 
Let us finally formulate the claimed bifurcation result.

\begin{theo}[Bifurcation]\label{t:EperBifurc}
Let $\Omega\subseteq \R^n$ be open and bounded with Lipschitz boundary, 
let $1<p\le 1+\frac{1}{n}$, let $\alpha>0$ such that \eqref{e:alph<<1} holds,
and let $\cEP$ satisfy \textup{(E1)-(E4)}. 
Moreover, let $(\lambda_{k,\alpha},u_{k,\alpha})_k$ be the eigensolutions of  
\eqref{e:E_per1Lap-a},\,\eqref{e:cG_Eper-a} from Theorem 
\ref{t:exist:CP1perLap} with corresponding critical values 
$c_{k,\alpha}$ and let $(\lambda_{k,1})_k$ be the eigenvalues of the
(unperturbed) 
$1$-Laplace operator according to \eqref{e:la1k1Lap}. Then we have
\[
\lim_{\alpha\to 0} \frac{c_{k,\alpha}}{\alpha} = \lim_{\alpha\to
  0}\lambda_{k,\alpha} = \lambda_{k,1} 
\]
for all $k\in\N$ and, hence, the eigenvalues $\lambda_{k,1}$ of the
unperturbed problem are bifurcation values of the perturbed problem 
\eqref{e:E_per1Lap-a},\,\eqref{e:cG_Eper-a}.
\end{theo}

\subsection{Perturbation of the constraint}
\label{ss:pertconst}

For $\beta>0$ and $p>1$ we now consider perturbed eigenvalue problems of the
$1$-Laplace operator of the form
\begin{equation}
 \cE_{TV}(v)\to \Min_{v\in L^p(\Omega)} \label{e:VP_Gpert1}
\end{equation}
subject to
\begin{equation}
\cG_1(v)+\cG_{\Per}(v)=\beta\,\label{e:VP_Gpert2}
\end{equation}
(cf.~\eqref{e:form6}, \eqref{e:form7}) where eigenfunctions had been defined
to be critical points of 
\[ \cE_{TV} + I_{\{\cG_1 + \cG_{Per}=\beta\}}  \]
on $L^p(\Omega)$. For the perturbation $\cGP:L^p(\Omega)\to\R$ we assume that
\bgl
\item[(G0)]
 $1<p<1+\frac{1}{n}$\,,
\item[(G1)]
 $\cG_{\Per}$ is locally Lipschitz continuous on $L^p(\Omega)$,
\item[(G2)]
 $\cGP$ is even, i.e. $\cGP(v)=\cGP(-v)$
for all $v\in L^p(\Omega)$,
\item[(G3)] there exists a constant $C_{\Per}>0$ such that for all $v\in
  L^p(\Omega)$ 
\begin{align}
 0\le \cG_{\Per}(v)&\le C_{\Per}\|v\|_p^p \,,\label{e:G_1est}
\end{align}
\item[(G4)] for all $v^*\in\partial \cGP(v)$ and all
$v\in L^p(\Omega)$ one has
\begin{align} 
  \|v^*\|_{p'}&\le p\,C_{\Per} \|v\|_p^{p-1}\,,\label{e:u^*le_u^p-1}
\end{align}
\item[(G5)] for all $v^*\in\partial \cGP(v)$ and all
$v\in L^p(\Omega)$ one has
\begin{align}
\left.
\begin{array}{ll}
v^*(x)>-1 & \text{ for } v(x)>0\\
v^*(x)<1 & \text{ for } v(x) <0
\end{array} \;\right\}
\quad \text{for a.e. } x\in \Omega\,.
\label{e:u^*u>-|u|}
\end{align}
\el
Condition (G5) is in particular needed for the (PS)-con\-di\-tion.
Clearly \eqref{e:u^*u>-|u|} is equivalent to 
\[
v^*(x)\frac{v(x)}{|v(x)|}>-1 \qquad \text{for a.e. $x\in \Omega$
with $v(x)\neq 0$}
\]
and, thus, \eqref{e:u^*u>-|u|} implies
\begin{align}
v^*(x)v(x) > -|v(x)|\label{e:u^*u>-|u|2} \qquad \text{for a.e. $x\in \Omega$
with $v(x)\neq 0$}\,.
\end{align}

Similar to the previous case, conditions (G1)-(G4) are satisfied in the case
of the Nemytskii potential  
\[ \cG_{\Per}(v)=\int_{\Omega}\int_0^{v(x)}g(x,s)\d s \d x \, \]
if the integrand $g$ satisfies (f1)-(f3) and, in addition, 
\begin{align} \label{e:results2}
G(x,t):=\int_0^t g(x,s)\d s \ge 0 
\quad\text{for all $t\in\R$ and a.e. $x\in\Omega$}\,
\end{align}
(cf.~Theorem \ref{t:cF}). For (G5) it is sufficient to require that
\begin{align} \label{e:results3}
 g(x,s) > -1 \quad\text{a.e. on $\Omega\times [0,\infty)$}\,, 
\end{align}
since $g(x,\cdot)$ is odd by (f3).  
It is not difficult to show that \eqref{e:results2} and convexity of 
$G(x,\cdot)$ imply $g\ge 0$ on $\Omega\times [0,\infty)$ 
and, thus, condition \eqref{e:u^*u>-|u|} is valid. Let us start with some
Euler-Lagrange equation as necessary condition for critical points. The proof
can be found in Section~\ref{ss:proofPertconstr} below.

\begin{theo}[Euler-Lagrange equation]\label{t:ELGGper}
Let $\Omega\subseteq\R^n$ be open and bounded with Lipschitz boundary, let
$1<p<1+\frac{1}{n}$, let $\cGP$ satisfy \textup{(G1)-(G5)}, and let  
$u$ be a critical point of variational problem
\eqref{e:VP_Gpert1},\,\eqref{e:VP_Gpert2} for some $\beta >0$.
Then there is some $s\in L^{\infty}(\Omega)$ with 
\[
s(x)\in \Sgn(u(x)) \quad\text{for a.e. }x\in\Omega\,,  
\] 
a vector field $z\in L^\infty(\Omega,\R^n)$ with
\[
\Div z\in L^{p'}(\Omega)\,,\quad \|z\|_\infty= 1\,,\quad \text{ and }\quad
\cE_{TV}(u)=-\int_\Omega \Div z\, u \d x\,, 
\]
some $u^*\in \partial \cGP(u)\subseteq L^{p'}(\Omega)$ and a Lagrange
multiplier $\lambda\in \R$ 
such that the Eu\-ler-La\-grange equation
\begin{align}
-\Div z = \lambda (s + u^*)\label{e:ELGGper}
\quad\text{on } \Omega
\end{align}
is satisfied. 
In the case where 
\[ \cG_{\Per}(v)=\int_{\Omega}\int_0^{v(x)}g(x,s)\d s \d x \, \]
with $g$ satisfying \textup{(f1)-(f3)}, \eqref{e:results2}, and
\eqref{e:results3}, we have  
\[ 
u^*(x)\in  \Big[\essinf_{s\to u(x)}g(x,s), \esssup_{s\to u(x)}g(x,s)\Big]\quad
\text{ for a.e. $x\in\Omega$\,.} 
\]
If, in addition, $g(x,\cdot)$ is continuous for a.e. $x\in\Omega$, then 
$u^*(x)=g(x,u(x))$ a.e. on $\Omega$ and \eqref{e:ELGGper} becomes
\[
 -\Div z = \lambda \big(s+g(x,u)\big) \qquad\text{on }\:\Omega\,.
\]
\end{theo}

\begin{rem}
With eigenfunction $u$ as test function in 
\eqref{e:ELGGper}, we get for the corresponding eigenvalue
\[
  \lambda = \frac{\cE_{TV}(u)}{\cG_1(u)+\<u^*,u\>}\,
\]
for some $u^*\in \partial \cGP(u)$.
In the unperturbed case we have $\lambda=\frac{\cE_{TV}(u)}{\beta}$ and the
eigenvalue is uniquely determined by the eigenfunction. However, 
for the perturbed problem it is not clear 
whether the eigenvalue $\lambda$ is uniquely 
determined by the eigenfunction~$u$ (cf.~also 
Remark \ref{r:results1a} about the perturbation of the energy).
\end{rem}

Next we formulate our main result about the existence of eigensolutions of the
perturbed problem \eqref{e:VP_Gpert1},\,\eqref{e:VP_Gpert2}.

\begin{theo}[Existence of eigensolutions]\label{t:CP_Gpert} 
Let $\Omega\subseteq\R^n$ be open and bounded with Lipschitz boundary, let
$1<p<1+\frac{1}{n}$, let $\beta>0$, and let $\cGP$ satisfy 
conditions \textup{(G1)-(G5)}.
Then there exists a sequence of eigenfunctions 
$(\pm u_{k,\beta})_k$ of the perturbed
eigenvalue problem \eqref{e:VP_Gpert1},\,\eqref{e:VP_Gpert2}
where the corresponding critical
values $c_{k,\beta}=\cE_{TV}(\pm u_{k,\beta})$ are characterized by 
\begin{align}
 c_{k,\beta} = \inf_{S\in \cS_k^\beta} \sup_{u\in S} \:\cE_{TV}(u)\,,\label{e:hatckb}
\end{align}
with 
\begin{align} 
  \begin{split}  
  \cS_k^\beta:=\{ S\subseteq L^p(\Omega) \text{ compact, symmetric}\setsep
  \hspace{43mm} \\
  \cG_1+\cG_{\Per}=\beta\text{ on }S,\ \gen_{L^p} S \ge k\}\,. \label{e:hatSkb}
\end{split}
\end{align}
The sequence of critical values $(c_{k,\beta})_k$ is unbounded.
For each $k\in\N$ the family of rescaled eigenfunctions 
$v_{k,\beta}:=\frac{u_{k,\beta}}{\|u_{k,\beta}\|_1}$ is bounded in $BV(\Omega)$ for
$\beta>0$ from bounded sets (in particular as $\beta \to 0$) and any choice of eigenfunctions $u_{k,\beta}$ having critical value 
$c_{k,\beta}$. 
Moreover, the Eu\-ler-La\-grange equation \eqref{e:ELGGper} from 
Theorem~\ref{t:ELGGper} holds for any critical point $u_{k,\beta}$. 
\end{theo}

The proof of Theorem \ref{t:CP_Gpert} is given in 
Section~\ref{ss:proofPertconstr}. It is again based on the general critical
point Theorem \ref{t:exCP} stated in Section \ref{ss:nonsmooth}. Finally we
formulate the intended bifurcation result. 

\begin{theo}[Bifurcation]\label{t:bifb}
Let $\Omega\subseteq\R^n$ be open and bounded with Lipschitz boundary, let
$1<p<1+\frac{1}{n}$, let $\beta>0$, and let $\cGP$ satisfy \textup{(G1)-(G5)}.
Moreover, let $(\lambda_{k,\beta},u_{k,\beta})_k$ be the eigensolutions of  
\eqref{e:VP_Gpert1},\,\eqref{e:VP_Gpert2} from Theorem 
\ref{t:CP_Gpert} with corresponding critical values 
$c_{k,\beta}$ and let $(\lambda_{k,1})_k$ be the eigenvalues of the
(unperturbed) 
$1$-Laplace operator according to \eqref{e:la1k1Lap}. Then we have
\begin{align}\label{e:ckb/b=lb=l}
\lim_{\beta\to 0} \frac{c_{k,\beta}}{\beta} = \lim_{\beta\to
  0}\lambda_{k,\beta} = \lambda_{k,1} 
\end{align}
for all $k\in\N$ and, hence, the eigenvalues $\lambda_{k,1}$ of the
unperturbed problem are bifurcation values of the perturbed problem 
\eqref{e:VP_Gpert1},\,\eqref{e:VP_Gpert2}.
\end{theo}

\section{Preparation of the proofs}
\label{s:tools}

Before we carry out the proofs of our main results, we provide some tools from
nonsmooth critical point theory and some essential norm estimates.

\subsection{Tools from nonsmooth critical point theory}
\label{ss:nonsmooth}

Our existence results for eigensolutions rely on 
nonsmooth critical point theory for merely lower semicontinuous
functionals based on the weak slope. With Theorem \ref{t:exCP} below 
we provide a modified version of the general Ljus\-ter\-nik-Schni\-rel\-man
 type 
theorem stated in Degiovanni \& Schuricht
\cite[Thm.~2.5]{degiovanni-s:98}. Though several 
similar results can be found in the literature, we did not find a direct
reference for the presented version. Therefore we give a self-contained proof
for the convenience of the reader and to  
keep track for some technical details. 
For completeness we first introduce the notion of weak slope. 

Let $(M,d)$ be a metric space, let $\cF:M\to \R$ be a continuous
functional, and let $B_\delta(u)\subseteq M$ 
be the open ball of radius $\delta$ around
$u$. The weak slope $|d\cF|(u)\in[0,\infty]$ at the point $u\in M$ is 
the supremum over all $\sigma\in [0,\infty)$ such that there exists some 
$\delta>0$ and a continuous function 
$\eta: B_\delta (u)\times [0,\delta]\to M$ with  
\begin{align*}
d(\eta(v,t),v)\le t \quad\text{ and }\quad 
\cF(\eta(v,t))\le \cF(v)-\sigma t
\end{align*}
for all $(v,t)\in B_\delta(u)\times [0,\delta]$. This notion extends
the value of $\|\cF'(u)\|$ for a smooth function $\cF$ to merely continuous
functions on a metric space (cf.~\cite{degiovanni-m:94}). 
We thus call $u\in M$ a critical point
of $\cF$ in the sense of the weak slope provided $|d\cF|(u)=0$. 

In a consistent way we extend the weak slope to a lower semicontinuous
function $\cF:M\to \R$ by means of the epigraph 
\[ 
\epi(\cF):=\{(v,t)\in M\times \R\setsep \cF(v)\le t\}
\]
equipped with the metric 
\begin{align}\label{e:epim}
  \tilde d\big((v_1,t_1),(v_2,t_2)\big):=\sqrt{d(v_1,v_2)^2 + (t_1-t_2)^2}\,.
\end{align}
Using the projection $\mathscr{G}_{\cF}:\epi\,(\cF)\to \R$ given by
$\mathscr{G}_{\cF}(v,t):=t$, we define
\[
|d\cF|(u):=
\begin{cases}
\frac{|d\mathscr{G}_{\cF}|(u,\,\cF(u))}
     {\sqrt{1-\big(|d\mathscr{G}_{\cF}|(u,\cF(u))\big)^2}}&
     \text{if }|d\mathscr{G}_{\cF}|(u,\,\cF(u))<1\,, \\ 
\infty &\text{if }|d\mathscr{G}_{\cF}|(u,\cF(u))=1\,.
\end{cases}
\] 
This way the weak slope of $\cF$ is traced back to the weak slope of
the continuous function $\mathscr{G}_{\cF}:\epi\,(\cF)\to \R$. 

In order to rule out possible critical points $(u,t)$ of 
$\mathscr{G}_{\cF}$ with $t>\cF(u)$, we assume the so-called 
epigraph (or short (epi)-) condition, i.e. for each $b>0$ we assume to have 
\begin{align}
\inf\, \big\{|d\mathscr{G}_{\cF}|(u,t)\setsep (u,t)\in \epi(\cF),\: \cF(u)<t,\:
|t|<b\big\}>0\,.\label{e:epi} 
\end{align}
Now we are able to state the general critical point theorem 
for even and lower semicontinuous functionals where we use the genus as
topological index.

\begin{theo}\label{t:exCP}
Let $X$ be a real Banach space and let $\cF:X\to \R\cup\{\infty\}$ be such
that
\bgl
\item[\textup{(F1)}]\label{i:Flsc+...} 
$\cF$ is lower semicontinuous, 
even (i.e. $\cF(u)=\cF(-u)$), and $F(0)=\infty$, 

\item[\textup{(F2)}]\label{i:Fge0} $\cF$ is bounded from below,

\item[\textup{(F3)}]\label{i:FPS} $\cF$ satisfies the $\textup{(PS)}$-condition,

\item[\textup{(F4)}] \label{i:Fepi}$\cF$ satisfies the
  $\textup{(epi)}$-condition, and 

\item[\textup{(F5)}] \label{i:existgen>k} for all $k\in \N$ there exists  
$\Phi :\S^{k-1}\to X$ bijective, continuous and odd (i.e. $\Phi(-x)=-\Phi(x)$)
with  
\[\sup\{\cF(\Phi(x))\setsep x\in \S^{k-1}\}<\infty\,\]
where $\S^{k-1}$ denotes the $(k-1)$-dimensional sphere in $\R^k$.
\el
Then there exists a sequence of pairs $\pm u_1, \pm u_2,\ldots$ of critical
points of $\cF$ with corresponding critical values $c_k=\cF(\pm u_k)$,
$k\in\N$, given by 
\begin{align}
c_k=\inf_{S\in \cS_k} \sup_{v\in S}\cF(v)\label{e:c_kgeneral}
\end{align}
where
\[
\cS_k:=\{S\subseteq X\setminus\{0\} \text{ symmetric and compact} \sets \gen
S\ge k\} \,.
\]

If the sublevel sets $\{\cF\le \gamma\}$ are compact for any
$\gamma\in\R$, then 
\[ c_k\to \infty \quad\text{as}\quad k\to \infty\,,  \] 
there is a set $S\in \cS_k$ attaining the infimum in \eqref{e:c_kgeneral}
\begin{align}  \label{e:tools1}
c_k=\sup_{v\in S}\cF(v)\,,
\end{align}
and any $S\in \cS_k$ with \eqref{e:tools1}  
contains a critical point $\tilde u_k$ with critical value $c_k$.
\end{theo}

The proof of the theorem will be given in Section \ref{ss:proofCP} below. 

\pagebreak[3]

\begin{rem}
\begin{enumerate}
\renewcommand\labelenumi{\textup{(\theenumi)}}

\item With minor modifications we could use Theo\-rem~2.5 from
  \cite{degiovanni-s:98} to get the existence of a sequence
of critical points as in the previous theorem. 
However, the unboundedness of the $c_k$ and their minimax characterization 
as in \eqref{e:c_kgeneral} is not stated there.
The unboundedness of the $c_k$ could probably also be derived from 
Degiovanni \& Marzocchi \cite[Thm.~3.10]{degiovanni-m:94}, but there the category is
used as topological index. Though it is well known 
that the genus of a closed symmetric set equals its category in
 the projective space where antipodal points are identified
(cf.~Rabinowitz \cite[Thm.~3.7]{rabinowitz:73} and Fadell
\cite[p.~40]{fadell:80}),  
critical point theory for merely lower semicontinuous functionals $\cF$ is
reduced to the investigation of the continuous functional $\sG_{\cF}$ where 
some rather technical arguments are needed to verify that the critical values
obtained with the concept of category agree with that obtained by
using the genus (cf.~Littig \& Schuricht \cite[Cor.~2.2]{littig-s:13}
and its proof).

\item The situation of the theorem might be covered by the abstract
results of Corvellec \cite{corvellec:97}, 
but it is not immediate and might be quite technical 
to deduce the desired statements.

\item If condition \textup{(F5)} is satisfied not for all $k\in\N$ but
  only for some $k_0\in\N$ (e.g. if $X$ is finite dimensional),
  it is not difficult to adapt our proof to show
  that there exist at least $k_0$ pairs of critical points $\pm u_1,\ldots,\pm
  u_{k_0}$ with corresponding critical values given by
  \eqref{e:c_kgeneral}.

\item 
Notice that, in general, there might be critical points of $\cF$ with critical
level $c_k$ that do not
belong to some $S\in \cS_k$ satisfying \eqref{e:tools1}.
\end{enumerate}
\end{rem}

\subsection{Norm estimates}
\label{ss:normest}

Here we derive some norm estimates needed for our 
convergence results.

\begin{prop}\label{p:norm1}
Let $\Omega\subseteq\R^n$ be open and bounded, let 
$p\in[1,\frac{n}{n-1}]$, and let $C_{\textup{BV}}$ be the embedding constant of
$W^{1,1}_0(\Omega)$ in $L^{\frac{n}{n-1}}(\Omega)$ (cf.~\eqref{e:int8}). 
Then we have
\[
 \|u\|_p\le
 C_{\textup{BV}}^{\tfrac{n(p-1)}{p}}\,\|u\|_1^{1-\tfrac{n(p-1)}{p}}\,
 \|Du\|_1^{\tfrac{n(p-1)}{p}} \quad\text{for all } u\in W^{1,1}_0(\Omega)\,.
\]
\end{prop}
\begin{proof}
For $u\in L^{\frac{n}{n-1}}(\Omega)$ the interpolation inequality tells us that
\[
\|u\|_p\le\|u\|_1^{\theta}\|u\|_{\frac{n}{n-1}}^{1-\theta}
\]
with
\[
 \frac{1}{p} = \frac{\theta}{1} + \frac{1-\theta}{n/(n-1)}
 \quad\text{or, equivalently,}\quad \theta= 1-\frac{n(p-1)}{p} \,.
\]
Then the assertion directly follows with \eqref{e:int8}.
\end{proof}

\pagebreak[3]

Consequently we can control $\|u\|_p$ by joint knowledge of $\|u\|_1$ and
$\|Du\|_1$. Since $\|Du\|_1=\cE_{TV}(u)$ for $u\in W^{1,1}_0(\Omega)$, the
following statement for $BV$-functions is not surprising. 

\begin{cor}\label{cor:||pp}
 Let $\Omega\subseteq\R^n$ be open and bounded with Lipschitz boundary, let
 $p\in[1,\frac{n}{n-1}]$,
let $C_{\textup{BV}}$ be the embedding constant of
$W^{1,1}_0(\Omega)$ in $L^{\frac{n}{n-1}}(\Omega)$ (cf.~\eqref{e:int8}), and let
$\cE_{TV}$ be as in \eqref{e:defETV}. Then
\[
 \|u\|^p_p\le
 C_{\textup{BV}}^{(p-1)n}\,\|u\|_1^{n-(n-1)p}\,\cE_{TV}(u)^{(p-1)n}
 \quad\text{for all } u\in BV(\Omega)\,.
\]
If additionally $p\le \frac{n+1}{n}$, we have
\begin{align}
\|u\|^p_p\le  C_{\textup{BV}}^{(p-1)n}\,\|u\|_1^{n-(n-1)p}\,\big(\cE_{TV}(u)
+1\big) \quad\text{for all } u\in BV(\Omega)\,. \label{e:ppbound} 
\end{align}
\end{cor}

\begin{proof}
The fist estimate follows by taking the $p$-th power of the inequality in 
Proposition~\ref{p:norm1} and by approximating $\cE_{TV}(u)$ as in 
Theo\-rem~3.1 of\cite{littig-s:13}. 
For the second estimate we observe that 
$t^{(p-1)n}\le 1+t$ for $t\ge 0$ by $(p-1)n\le 1$ and then we set 
$t=\cE_{TV}(u)$. 
\end{proof}
Notice that \eqref{e:ppbound} allows to control the $p$-th order growth of
$\|u\|_p^p$ by the first order growth of $\cE_{TV}(u)$ provided $\|u\|_1$ is
known to be bounded.

\section{Proofs of the main results}
\label{s:proofs}

We first present the proof of Theorem \ref{t:cF} about  
properties of integral functionals we have in mind as perturbations. 
Then the general Theorem \ref{t:exCP} about existence of critical points
is verified. In Section \ref{ss:proofpertenergy} proofs related to 
perturbations of the energy are given and, finally, Section
\ref{ss:proofPertconstr} collects the proofs related to 
perturbations of the constraint.

\subsection{Proof of Theorem~\ref{t:cF}}
\label{ss:proofcF}

\begin{proof}[Proof of Theorem~\ref{t:cF}] 
If $\cF$ is well-defined, then
antisymmetry of $f(x,\cdot)$ as in \eqref{e:antsymf} implies 
that $\cF$ is symmetric, i.e.\ $\cF(u)=\cF(-u)$, and we have 
\begin{align}\label{e:proofs1}
\cF(u)=\int_\Omega\int_0^{u(x)}f(x,s)\d s \d x = \int_\Omega
\int_0^{|u(x)|}f(x,s)\d s\d x=\cF(|u|)\,. 
\end{align}

Let us now verify that $\cF$ is well-defined. 
By (f1) function $F:\Omega\times[0,\infty)\to\R$ with 
\[ F(x,t):=\int_0^tf(x,s)\d s \]
is a Carath\'eodory function (which includes that $F$ is well defined)
and, hence, $F(\cdot,u(\cdot))$ is measurable on $\Omega$ for any 
measurable $u$.

We now take $u, w\in L^p(\Omega)$ with $\|u\|_p,\,\|w\|_p\le R$. Then 
$(|u|+|w|)^{p-1}\in L^{p'}(\Omega)$ with 
\[
 \|(|u|+|w|)^{p-1}\|_{p'}^{p'} = \|(|u|+|w|)\|_p^p\le (2R)^p.
\]
By \eqref{e:proofs1}, \eqref{e:fleCsp-1}, and H\"older's inequality we get
\begin{align}
|\cF(u)-\cF(w)|& = \left|\int_\Omega \int_{|w(x)|}^{|u(x)|} f(x,s) \d s \d
  x\right|\notag\\ 
&\le p\,C_{\Per}\int_\Omega\big||u(x)|-|w(x)|\big|\,
\big(|u(x)|+|w(x)|\big)^{p-1}\d x\notag\\
& \le p\,C_{\Per}\|u-w\|_p\, \|(|u|+|w|)^{p-1}\|_{p'}\notag\\
&\le p\,C_{\Per} (2R)^{p-1}\|u-w\|_p\,.\label{e:Lip_est}
\end{align}
Since $\cF(w)=0$ for $w=0$, we readily obtain that $\cF(u)$ is finite for
all $u\in L^p(\Omega)$. Moreover $\cF$ is uniformly Lipschitz continuous on
bounded subsets of $L^p(\Omega)$.  

A straightforward calculation using \eqref{e:fleCsp-1} gives
for $\alpha\ge 0$ that
\begin{align*}
 |\cF(\alpha u)|&\le \int_\Omega \left|\int_0^{\alpha u(x)} f(x,s) \!\d
   s\right|\d x\\ 
&\le \int_\Omega\left| \int_0^{|\alpha u(x)|} p\,C_{\Per}s^{p-1} \d
  s\right|\!\d x\displaybreak[3]\\ 
&= \int_\Omega C_{\Per}|\alpha u|^p\d x\\
&= C_{\Per}\,\alpha ^p \|u\|_p^p\,,
\end{align*}
i.e. we have shown \eqref{e:cFleC|u|_p^p}.

It remains to prove assertion (4) about $\partial\cF(u)$.
For $u,v\in L^p(\Omega)$, $u^*\in\partial \cF(u)$, and with the notation
\begin{align}
F_x(t):=F(x,t)\,,\label{e:Fx=F(x,.)}
\end{align}
we derive
\begin{align*}
\int_\Omega u^* v \d x \le \cF^0(u;\, v)&=\limsup_{w\to u,\,t\downarrow0}\:
\frac{\cF(w+t v)-\cF(w)}{t}\\&=\limsup_{w\to u,\,t \downarrow0}\:
\int_\Omega\frac{F_x(w(x)+t v(x))-F_x(w(x))}{t}\d x\,. 
\end{align*}
Notice that $F_x$ is the primitive of a locally bounded function
for a.e.~$x\in \Omega$ by~\eqref{e:fleCsp-1}. Hence 
we are in the situation of Example~2.2.5 from \cite{clarke:87} and obtain that
$F_x$ is locally Lipschitz continuous with 
\begin{align}\label{e:proofs2}
\partial F_x(t)=\Big[\essinf_{s\to t}f(x,s),\,\esssup_{s\to t}
f(x,s)\Big] \quad\text{for }t\in\R\,.
\end{align}
Again by \eqref{e:fleCsp-1} we get
\begin{align}
|F^*|\le p\,C_{\Per}|t|^{p-1} \quad\text{for all }
F^*\in \partial F_x(t)\,. \label{e:partialFx}
\end{align}

Let us now choose a sequence $(w_k)_k$ with $w_k\to u$ in $L^p(\Omega)$ and 
$t_k\downarrow 0$ with $t_k\le 1$ such that
\begin{align}
\cF^0(u;\, v)=\lim_{k\to \infty}
\int\frac{F_x(w_k(x)+t_kv(x))-F_x(w_k(x))}{t_k}\d x\,.\label{e:cF0=lim} 
\end{align}
Without loss of generality we may assume that $w_k(x)\to u(x)$ a.e. on
$\Omega$. By Lebourg's Theorem (cf.~\cite[Thm.~2.3.7]{clarke:87})
we have that for a.e. $x\in\Omega$ and every $k\in \N$ there is some
$\theta\in (0,1)$ and $F_k^*(x)\in \partial F_x\big(w_k(x)+\theta
t_kv(x)\big)$ such that 
\begin{align}
\big|F_x\big(w_k(x) + t_k v(x)\big)-
F_x\big(w_k(x)\big)\big|&=\big|F_k^*(x)\big(w_k(x) + t_k v(x)
-w_k(x)\big)\big|\label{e:Fw_kt_k}\\ 
&=t_k\big|F_k^*(x) v(x)\big|\notag\\
&\le p\, C_{\Per}(|w_k(x)|+|v(x)|)^{p-1}\,|v(x)|\,\notag
\end{align}
by \eqref{e:partialFx}. Obviously 
\[ h_k:=|w_k|+|v|\to |u|+|v| \quad\text{in }L^p(\Omega)\,\]
and the nonlinear operator $J_p: L^p(\Omega)\to L^{p'}(\Omega)$
given by 
\[ J_p(u)(x):=|u(x)|^{p-1}\sgn(u(x)) \]
is a homeomorphism (cf.~\cite[p.~72]{cioranescu:90}). 
Thus 
\[
h_k^*:=J_p(h_k) \to J_p(|u|+|v|) \quad \text{in }L^{p'}(\Omega) \,.
\]
Whence 
\[ g_k^{\phantom{*}}:=h_k^*\,|v| \to g:=(|u|+|v|)^{p-1}|v| \quad\text{in }
L^1(\Omega) 
\]
and, by assumption, also pointwise a.e. on $\Omega$. 
Picking an appropriate subsequence if necessary we may assume that
$\sum_{k\in\N}\|g_k-g\|_1<\infty$. Then  
$g+\sum_{k\in\N}|g_k|$ is a majorant of all $g_k$ and also of
all integrands in \eqref{e:Fw_kt_k}. Therefore, by Fatou's Lemma, 
\eqref{e:cF0=lim} implies 
\[
\int_\Omega u^*(x) v(x)\d x\le 
\cF^0(u;\, v)\le \int_\Omega \limsup_{k\to \infty}
\frac{F_x(w_k(x)+t_kv(x))-F_x(w_k(x))}{t_k}\d x\,. 
\]
Note that the integrand on the right hand side is
bounded by $F_x^0(u(x);\, v(x))$ for a.e. $x\in \Omega$. Since the argument
holds true for all $v\in L^p(\Omega)$, we can choose  
$v=t\chi_E$ for appropriate $E\subseteq\Omega$ and $t\in\R$ to obtain 
\[
u^*(x)\,t\le F^0_x(u(x);\, t) 
\quad\text{for all $t\in\R$ and a.e. $x\in\Omega$\,.}
\]
Consequently, by definition,
\[
u^*(x)\in \partial F_x(u(x)) \quad\text{for a.e. } x\in\Omega\,
\]
and with \eqref{e:proofs2} we have verified \eqref{e:u*f-f+}. 
By \eqref{e:fleCsp-1} we thus obtain
\begin{align*}
\|u^*\|_{p'}^{p'}&\le \int_\Omega 
\max\bigg\{\Big|\essinf_{s\to u(x)}f(x,s)\Big|, 
\Big|\esssup_{s\to u(x)}f(x,s)\Big|\bigg\}^{p'}\d x \\
&\le \int_\Omega (p\,C_{\Per})^{p'} |u(x)|^p\d x\\
&= (p\,C_{\Per})^{p'} \|u\|_p^p
\end{align*}
and \eqref{e:|u*|p'} follows.
\end{proof}

\subsection{Proof of Theorem~\ref{t:exCP}}
\label{ss:proofCP}

\begin{proof}[Proof of Theorem~\ref{t:exCP}]It is well known that
(F5) ensures the classes $\cS_k$ to be nonempty
  (cf.~\cite[Chap.~44.3]{zeidler:84}) and $c_k<\infty$. Thus, 
by boundedness of $\cF$ from
below, the values $c_k$ in \eqref{e:c_kgeneral} are finite. 

If a set $\ti S\subseteq \epi(\cF)$ has the property that $(u,s)\in \ti S$ implies $(-u,s)\in \ti S$, we define the genus $\gen_1 \ti S$ of $\ti S$
as the genus of the projection of $\ti S$ on the first coordinate, i.\,e. 
\[
\gen_1\ti S:=\gen \{u\in X\setsep (u,s)\in \ti S\}\,.
\]
Taking
\[
\ti \cS_k:=\{\ti S \subseteq \epi(\cF)\sets \ti S\text{ compact}\,, \  
\gen_1 \ti S \ge k\,, \ (-u,s)\in \ti S \;\; \forall (u,s)\in \ti S \} \,,
\]
\[
\ti c_k:= \inf_{\ti S\in \ti\cS_k}\sup_{(u,s)\in \ti S}\sG_{\cF}(u,s)
\]
we have
\[
c_k=\ti c_k\,.
\]
Indeed, invoking the definition of $\sG_{\cF}$ and $\ti S\subseteq \epi(\cF)$,
we see that the value $\ti c_k$ does not change if we 
restrict our attention to sets $\ti S\in \ti \cS_k$
of the form 
\[
\ti S= S\times\big\{\sup_{u\in S}\cF(u)\big\}
\]
with $S\in \cS_k$. We may assume that $\sup_{u\in S}\cF(u)<\infty$
by (F5) and, hence, for those sets $\ti S$ the equality is
immediate. 

We define the set of critical points of $\cF$ at level $c$ by
\[
K_c:=\{u\in X\sets \cF(u)=c\text{ and } |d\cF|(u)=0\}\,.
\]
Let us assume that $c_k$ is not a critical value, i.e.~$K_{c_k}=\emptyset$.
We will show that then there is some $\ti \varepsilon>0$ such that
\begin{align}\label{e:proofs3}
  K_c=\emptyset \quad \text{ for all } \quad 
  c\in (c_k-\ti \varepsilon, c_k + \ti\varepsilon)\,.
\end{align}
If this is not true, we find a sequence of critical points $(u_j)_j$ of the function $\cF$
with $\cF(u_j)\to c_k$. Then, by definition, $(u_j, \cF(u_j))_j$ is a sequence
of critical points of the continuous function
$\sG_{\cF}: \epi(\cF)\to \R$. Since $(u_j)_j$ is a
Pa\-lais-Smale sequence for $\cF$, it admits a convergent subsequence (denoted
the same way) with $u_j\to:u$. By lower
semicontinuity of $\cF$ we have $\cF(u)\le c_k$. 
Since the weak slope is lower
semicontinous with respect to the graph metric (see
\cite[Prop.~2.6]{degiovanni-m:94}), we obtain that $(u, c_k)=\lim_{j\to
  \infty} (u_j, \cF(u_j))$ is a critical point of $\sG_{\cF}$. From the
(epi)-con\-di\-tion \eqref{e:epi} we derive that $c_k=\cF(u)$ and, therefore,
$u \in K_{c_k}$. But this is a contradiction and verifies \eqref{e:proofs3}.

According to the first part of the proof of Theo\-rem~2.5 in
\cite{degiovanni-s:98} (applied with $f=\sG_{\cF}$, $X=\epi(\cF)$,
$\Phi(u,s)=(-u,s)$, $\mathcal{O}=\emptyset$) there is some 
$\varepsilon\in(0,\ti\varepsilon]$ and a continuous map $\eta :
\epi(\cF)\times [0,1]\to \epi(\cF)$ such that for all $(u,s)\in\epi(\cF)$, 
all $t\in[0,1]$, and with the epigraph metric $\tilde d$ as in \eqref{e:epim} 
\begin{align}
\tilde d\big(\eta((u,s),t), (u,s)\big)&\le t\notag\\
s\not \in [c_k-\ti \varepsilon, c_k + \ti \varepsilon]\ &\Rightarrow \ \eta \big((u,s),t\big)=(u,s)\notag\\
\eta\big(\big\{\sG_{\cF}\le c_k + \varepsilon\big\}, 1\big) &\subseteq \big\{\sG_{\cF}\le c_k -\varepsilon\big\}\label{e:eta1}\\
\eta\big((-u,s) ,t\big) &= -\eta ((u,s), t)\,.\label{e:eta2}
\end{align}

By \eqref{e:c_kgeneral} there is $S_1\in \cS_k$ such that
\[
\sup_{u\in S_1}\cF(u)\le c_k + \varepsilon\,.
\]
For $a:=\sup\nolimits_{u\in S_1}\cF(u)$ we define $\eta_1: X\to X\times \{a\}$ by
\[
\eta_1(u):=(u,a)\quad \text{and}\quad 
T_1:=\eta_1(S_1)=S_1\times\{a\}\subseteq\epi(\cF)\,.
\]
With $\eta$ from above we consider
\[
T_2:=\eta(T_1,1)\subseteq \epi(\cF)\,.
\]
By \eqref{e:eta1} we have $s\le c_k-\varepsilon$ for all $(u,s)\in T_2$.
Let $\eta_2: \epi(\cF)\to X$ denote the projection given by
\begin{align}\label{e:proofs4b}
\eta_2(u,s):= u\quad\text{and with}\quad S_2:=\eta_2(T_2)\,,
\end{align}
we then obtain
\begin{align}   \label{e:proofs4}
\sup_{u\in S_2}\cF(u) \le \sup_{(u,s)\in T_2} s \le c_k-\varepsilon\,.
\end{align}
The set $S_2$ is obtained as continuous image of $S_1$ under
$\eta_2\circ\eta\circ\eta_1$ and thus compact.
By \eqref{e:eta2} we see that 
$\eta_2\circ\eta\circ\eta_1$ is odd and, thus, 
an elementary property of genus gives 
\[
\gen S_2\ge \gen S_1 \,.
\]
Consequently, $S_2\in \cS_k$ and \eqref{e:proofs4} contradicts 
the definition \eqref{e:eta1} of $c_k$. Hence our assumption 
$K_{c_k}=\emptyset$ must be wrong and $c_k$ has to be a critical level for any
$k\in\N$.

For the proof of the remaining assertions let 
$\{\cF\le\gamma\}$ be compact for any $\gamma\in\R$.
Here we also use a compactness result of Blaschke  
(cf.~\cite[Thm.~4.4.15]{ambrosio-t:04}) saying that the set $\cK$ 
of nonempty compact subsets of a compact metric space $(K,d)$ 
is compact provided $\cK$ is equipped with the Hausdorff distance  
\[
  d_H(K_1,K_2):= \sup_{x\in K_1} d(x,K_2) + \sup_{x\in K_2} d(x, K_1)\,.
\]
Moreover, if $K_j\to K_0$ in the Hausdorff distance, then $x_0\in K_0$ 
if and only
if for each $j\in \N$ there is some $x_j\in K_j$ such that $x_j\to x_0$ 
(cf.~\cite[Prop.~4.4.14]{ambrosio-t:04}).

First we fix $k\in\N$ and choose a sequence $(S_j)$ in $\cS_k$ with 
\[
c_k=\lim_{j\to \infty} \sup_{v\in S_j}\cF(v)\,.
\]
We can assume that all $S_j$ belong to the compact
set $\{\cF\le c_k+1\}$ and, by Blaschke's theorem, that they converge 
to some compact 
$S\subseteq \{\cF\le c_k+1\}$ with respect to the Hausdorff metric.
The pointwise characterization of the limit and the 
lower semicontinuity of $\cF$ imply that $S$ is symmetric, that 
$0\not\in S$ (recall $\cF(0) = \infty$), and that
\[
\sup_{v\in S}\cF(v)\le \lim_{j\to \infty} \sup_{v\in S_j}\cF(v)=c_k\,.
\]
By a standard property of genus there is an open neighborhood 
$U$ of $S$ with $\gen S = \gen \overline{U}$
(cf.~\cite[Chap.~44.3]{zeidler:84}). The convergence $S_j\to S$ in the
Hausdorff metric implies $S_j\subseteq U$ for $k$ large enough. 
Hence, the monotonicity of genus with respect to inclusions gives
\[
  \gen S \ge \limsup_{j\to\infty} \gen S_j \ge k\,.
\]
Therefore $S\in \cS_k$ and the definition of $c_k$ implies \eqref{e:tools1}.

For fixed $k$ we now choose any $S\in \cS_k$ satisfying 
$c_k=\sup_{v\in S}\cF(v)$ and let us assume that 
\begin{align}   \label{e:proofs4a}
  S\cap K_{c_k}=\emptyset\,.
\end{align}
We show that there exists a neighborhood $\cU$ of $S\times\{c_k\}$ in 
$\epi(\cF)$ containing no critical points of $\sG_{\cF}$. Otherwise
we find critical points $(v_j,t_j)$ of $\sG_{\cF}$ with $(v_j,t_j)\to (v,c_k)$
for some $v\in S$ (recall compactness of $S$ and thus $S\times \{c_k\}$). Since the weak slope is
lower semicontinuous, $(v,c_k)$ is a critical point of $\sG_{\cF}$ and,
by the (epi)-con\-di\-tion (F4), $v\in S$ is critical point of $\cF$ with
critical value $c_k$. But this contradicts \eqref{e:proofs4a} and verifies our
claim. Consequently, by the compactness of $S\times \{c_k\}$, there is some open neighborhood 
$\cO$ of the critical points of $\sG_\cF$ in $\epi(\cF)$ with 
\[
\cO\cap (S\times \{c_k\})=\emptyset\,.
\]
According to Deformation Theorem 2.14 in \cite{corvellec-dm:93},
applied at critical value $c_k$, there exists a continuous map $\varphi:
\epi(\cF)\times [0,1]\to \epi(\cF)$ and some $\varepsilon>0$ such that 
\[
\sG_{\cF}(\varphi((u,s),1))\le c_k-\varepsilon
\quad\text{for } (u,s)\in \{\epi(\cF)\setminus\cO \sets
\sG_{\cF}(u,s)\le c_k+\varepsilon \}
\]
An easy adaption of the proof of \cite[Thm.~2.17]{corvellec-dm:93} shows
that we can assume
\[
\varphi((u,s),t)=-\varphi((-u,s),t) \quad \text{for }
(u,s)\in\epi(\cF),\ t\in[0,1]\,.
\]
With $\eta_2$ from \eqref{e:proofs4b} we get that
\[
\ti S:=\eta_2(\varphi(S\times\{c_k\},1))
\]
is symmetric and, as continuous image of a compact set, compact. Moreover
\begin{align}\label{e:proofs4c}
\cF(u)\le c_k-\varepsilon \qquad \text{for }u\in \ti S\,.
\end{align}
We have $\gen \ti S\ge \gen S\ge k$, since a continuous map does not decrease 
the genus. Thus $\ti S\in \cS_k$ and \eqref{e:proofs4c} contradicts the
definition of $c_k$. Consequently \eqref{e:proofs4a} must be wrong and $S$
contains a critical point $\ti u_k$ with critical value $c_k$.  

Finally let us assume that 
\[
c:=\limsup_{k\to\infty}c_k<\infty\,.
\]
According to \eqref{e:tools1} we can choose 
$S_k\in \cS_k$ with $c_k=\sup_{v\in S_k}\cF(v)$. Since the $c_k$ are
increasing, we can assume that all $S_k$ 
belong to the compact set $\{\cF\le c\}$ and, by Blaschke’s Theorem, 
that the $S_k$
converge to some compact and symmetric set $\ti S\subseteq\{\cF\le c\}$ 
in the Hausdorff metric. In particular $0\not\in\ti S$ by (F1). 
As above there is an open neighborhood $V$ of $\ti S$ with 
$\gen \overline{V} =\gen\ti S$. 
Since $S_k\subseteq V$ for $k$ large enough, the monotonicity of genus 
with respect to inclusions implies 
\[
k\le \gen S_k\le \gen \overline{V}= \gen \ti S \quad\text{for all }
k\in\N\,.
\]
But this contradicts the fact that the genus of a compact set is finite. 
Therefore $(c_k)_k$ cannot be bounded and the proof is complete. 
\end{proof}

\subsection{Proofs for perturbations of the energy}
\label{ss:proofpertenergy}

\begin{proof}[Proof of Theorem~\ref{t:ELG}]
We will apply \cite[Cor.~3.7]{degiovanni-s:98}~ with 
\[ Y=L^p(\Omega)\,,\; f_0=\cE_{TV}\,,\; f_1=\cE_{\Per}\,,\; 
g_0=-1\,,\; g_1=\cG_1-\alpha\,.
\]
 Let $u\in
BV(\Omega)$ with $\cG_1(u)=\alpha$. In order to prove the (epi)-con\-di\-tion
(cf.~\cite[Thm.~3.4]{degiovanni-s:98}) we need to show that there are
$u_1$, $u_2\in BV(\Omega))$ such that 
\[
\cG_1^0(u;\, u_1 -u) < 0\quad\text{and}\quad \cG_1^0(u;\, u-u_2)<0\,.
\]
Recalling the generalized gradient of $\cG_1$
(cf.~\cite[Prop.~4.23]{kawohl-s:07}) and using
\cite[Prop.~2.1.2]{clarke:87} with $u_1=0$ and $u_2=2u$, 
we derive 
\begin{align}
\cG_1^0(u;\, u_1-u) & =\cG_1^0(u;\, u - u_2)=\cG_1^0(u;\,-u) \notag\\
& =\max_{u^*\in\partial\cG_1(u)}\<u^*,-u\>=-\alpha<0.\label{e:cG0} 
\end{align}
The Eu\-ler-La\-grange equation \eqref{e:ELGEper} is now a consequence of
\cite[Prop.~4.23]{kawohl-s:07} and, for 
\eqref{e:results1a},
we use the properties of $\partial\cE_{\Per}$ stated in 
Theorem~\ref{t:cF}.
\end{proof}

\medskip

\begin{proof}[Proof of Proposition \ref{p:PS-Eper}]
Let $c\in\R$ and let $(u_j)_j$ be a Palais-Smale sequence for the function $\cE$, 
i.e. $\cE(u_j)\to c$ and $|d\cE|(u_j)\to 0$. 
In the case (E5') where $\cE_{\Per}$ is bounded from below by some $L\le 0$,
we eventually have 
\[
 \cE_{TV}(u_j)\le c + 1 - \cE_{\Per}(u_j)\le c+1-L\,.
\]
Since $\cE_{TV}$ is a norm on $BV(\Omega)$ equivalent to the standard norm, 
$(u_j)_j$ is bounded in $BV(\Omega)$. Thus, the compact embedding 
$BV(\Omega)\hookrightarrow L^p(\Omega)$ ensures the existence of a convergent
subsequence in $L^p$ and the (PS)-con\-di\-tion is verified.

If condition (E5'') is satisfied, we use \eqref{e:ppbound} and $\|u_j\|_1=\alpha$ to estimate
\begin{align}
c + 1&\ge  \cE_{TV}(u_j) + \cE_{\Per}(u_j)\notag\\
&\ge \cE_{TV}(u_j) - C_{\Per}\|u_j\|_p^p\notag\\
&\ge \cE_{TV}(u_j) - C_{\Per}\big(C_{\textup{BV}}^{(p-1)n}\,\|u_j\|_1^{n-(n-1)p}\,\cE_{TV}(u_j) + C_{\textup{BV}}^{(p-1)n}\,\|u_j\|_1^{n-(n-1)p}\big)\notag\\
&=\cE_{TV}(u_j) - C_{\Per} C_{\textup{BV}}^{(p-1)n}\,\alpha^{n-(n-1)p}\,\cE_{TV}(u_j) -C_{\Per}C_{\textup{BV}}^{(p-1)n}\,\alpha^{n-(n-1)p}\notag\\
&=\cE_{TV}(u_j)\big(1-  C_{\Per} C_{\textup{BV}}^{(p-1)n}\,\alpha^{n-(n-1)p}\big) - C_{\Per}C_{\textup{BV}}^{(p-1)n}\,\alpha^{n-(n-1)p}\,.\label{e:||pp_estim}
\end{align}
By \eqref{e:alph<<1} we obtain
\[
 \cE_{TV}(u_j)\le \frac{c + 1 +
   C_{\Per}C_{\textup{BV}}^{(p-1)n}\,\alpha^{n-(n-1)p}}{1-C_{\Per}
   C_{\textup{BV}}^{(p-1)n}\,\alpha^{n-(n-1)p}}\,. 
\]
Whence, as above, $(u_j)_j$ is bounded in $BV(\Omega)$ and 
there is a convergent subsequence in $L^p(\Omega)$. 
\end{proof}

\medskip

\begin{proof}[Proof of Theorem \ref{t:exist:CP1perLap}]
We will apply Theo\-rem~\ref{t:exCP} to 
\[ \cF=\cE_{TV}+\cE_{\Per} + I_{\{\cG_1=\alpha\}}\,. \]
Obviously, (F1) is satisfied. 
(F2) is clearly satisfied in the case (E5') where $\cE_{\Per}$ is
bounded from below. In the case (E5'') we have $p\le 1 +\frac{1}{n}$ and,
similar  
to \eqref{e:||pp_estim}, we use (E3), \eqref{e:ppbound}, and
\eqref{e:alph<<1} to derive for $v\in
BV(\Omega)$ with $\cG_1(v)=\alpha$ that 
\begin{align}
\cF(v) &= \cE_{TV}(v) + \cE_{\Per}(v) \nonumber\\
&\ge \cE_{TV}(v) - C_{\Per}\|v\|_p^p \nonumber\\
&\ge \big(1-C_{\Per}C_{\textup{BV}}^{(p-1)n}\alpha^{n-(n-1)p}\big)\,\cE_{TV}(v)
- C_{\Per}C_{\textup{BV}}^{(p-1)n}\alpha^{n-(n-1)p}\label{e:proofs5}\\
&\ge - C_{\Per}C_{\textup{BV}}^{(p-1)n}\alpha^{n-(n-1)p}\,. \nonumber
\end{align}
Hence $\cF$ is bounded from below and we have (F2) also in the second case. 
The function $\cF$ satisfies the (PS)-con\-di\-tion by Pro\-po\-si\-tion~\ref{p:PS-Eper}.
The (epi)-con\-di\-tion follows from \eqref{e:cG0}
(cf.~\cite[Thm.~3.4]{degiovanni-s:98}). 
In order to verify (F5) 
we choose linearly independent $v_1,\ldots v_k\in \Ccinfty$ 
and a desired map $\Phi:\S^{k-1}\to L^p(\Omega)$ is obviously given by 
\[
\Phi(x)=\Phi(x_1,\ldots, x_k) = \frac{\alpha\sum_{j=1}^n x_j
  v_j}{\big\|\sum_{j=1}^n x_j v_j\big\|_1}\,. 
\]
The existence of a sequence of eigensolutions now follows from
Theo\-rem~\ref{t:exCP}. 

For the unboundedness of the critical values 
$(c_{k,\alpha})_k$ we still need the compactness of the  
sublevel sets $\{\cF\le\gamma\}$. In the case of (E5') there is some
$\ti\gamma\in\R$ with $\cEP(v)\ge\ti\gamma$ for all $v\in L^p(\Omega)$.  
Hence
\[
  0\le \cE_{TV}(v) = \cF(v)-\cEP(v)\le \gamma-\ti\gamma \quad\text{for all }
  v\in \{\cF\le\gamma\}\,.
\]
Since $\cE_{TV}$ is an equivalent norm on $BV(\Omega)$, the set
$\{\cF\le\gamma\}$ is bounded in $BV(\Omega)$ and, by the
compact embedding $BV(\Omega)\hookrightarrow L^p(\Omega)$,
it is compact in $L^p(\Omega)$. For the second case (E5'') we argue
analogously using \eqref{e:proofs5} and \eqref{e:alph<<1}.

For the assertion concerning the Euler-Lagrange equation we can
obviously apply Theo\-rem~\ref{t:ELG}.
It remains to show the boundedness statement for the rescaled family 
$v_{k,\alpha}=\frac{u_{k,\alpha}}{\|u_{k,\alpha}\|_1}$.
This part of the proof is postponed to the end of this section.
\end{proof}

\medskip

\begin{proof}[Proof of Theorem \ref{t:EperBifurc}]
Using \eqref{e:Eper_le_||p}, \eqref{e:hatck}, \eqref{e:ppbound}, 
$\cS_k^\alpha$ according to \eqref{e:cSkalpha}, and 
\[
\hat c_{k,\alpha}:=\frac{c_{k,\alpha}}{\alpha}\,,
\]
we have 
\begin{align*}
\limsup_{\alpha\to 0} \,\hat c_{k,\alpha}
&= \limsup_{\alpha\to0} \inf_{S\in\cS^\alpha_k}\sup_{u\in S}
\frac{1}{\alpha}\big(\cE_{TV}(u) + \cE_{\Per}(u)\big)\\ 
&= \limsup_{\alpha\to0}\inf_{S\in\cS^1_k}\sup_{u\in S} \big(\cE_{TV}(u) +
\tfrac{1}{\alpha}\cE_{\Per}(\alpha u)\big)\displaybreak[2]\\ 
 &\le  \limsup_{\alpha\to0}\inf_{S\in\cS^1_k}\sup_{u\in S} \big(\cE_{TV}(u) +
 \tfrac{1}{\alpha}C_{\Per}\|\alpha u\|^p_p\big)\displaybreak[2]\\ 
 &=  \limsup_{\alpha\to0}\inf_{S\in\cS^1_k}\sup_{u\in S} \big(\cE_{TV}(u) +
 \alpha^{p-1}C_{\Per}\|u\|^p_p\big)\displaybreak[2]\\ 
 &\le \limsup_{\alpha\to0}\inf_{S\in\cS^1_k}\sup_{u\in S}\:\cE_{TV}(u)
 \big(1\!+\alpha^{p-1}C_{\textup{BV}}^{(p-1)n}C_{\Per}\big)\! +
 \alpha^{p-1}C_{\textup{BV}}^{(p-1)n}C_{\Per}\displaybreak[2]\\ 
 &=\inf_{S\in\cS^1_k}\sup_{u\in S} \cE_{TV}(u) = \lambda_{k,1}\,.
\end{align*}
Similarly, we obtain the reverse inequality by
\begin{align*}
\liminf_{\alpha\to 0} \hat c_{k,\alpha} &=
\liminf_{\alpha\to0}\inf_{S\in\cS^1_k}\sup_{u\in S} \big(\cE_{TV}(u) +
\tfrac{1}{\alpha}\cE_{\Per}(\alpha u)\big)\displaybreak[2]\\ 
&\ge  \limsup_{\alpha\to0}\inf_{S\in\cS^1_k}\sup_{u\in S} \big(\cE_{TV}(u) -
\tfrac{1}{\alpha}C_{\Per}\|\alpha u\|^p_p\big)\displaybreak[2]\\ 
&\ge \limsup_{\alpha\to0}\inf_{S\in\cS^1_k}\sup_{u\in S}\:\cE_{TV}(u) 
\big(1\! -\alpha^{p-1}C_{\textup{BV}}^{(p-1)n}C_{\Per}\big)\! -
\alpha^{p-1}C_{\textup{BV}}^{(p-1)n}C_{\Per}\displaybreak[2]\\ 
&=\inf_{S\in\cS^1_k}\sup_{u\in S} \:\cE_{TV}(u) = \lambda_{k,1}\,.
\end{align*}
Hence
\begin{align}   \label{e:proofs16}
 \lim_{\alpha\to 0} \frac{c_{k,\alpha}}{\alpha} = \
 \lim_{\alpha\to 0} \hat c_{k,\alpha} = \lambda_{k,1}\,,
\end{align}
i.e. the fist assertion is verified. 
The other limit will follow from the next proposition. 

\medskip

\begin{prop}\label{p:20}
Let $\Omega\subseteq \R^n$ be open and bounded with Lipschitz boundary,
let $\alpha>0$, let $1<p\le 1+\frac{1}{n}$,
and let $\cEP$ satisfy \textup{(E1)-(E4)}. 
Moreover, $(u_\alpha)_{\alpha>0}$ be a family of
critical points of \eqref{e:E_per1Lap-a},\,\eqref{e:cG_Eper-a}
with corresponding eigenvalues $(\lambda_\alpha)_\alpha$ such that the 
\[
\hat c_\alpha:=
\frac{1}{\alpha}\big(\cE_{TV}(u_\alpha)+\cE_{\Per}(u_\alpha)\big)
\]
are bounded for $0<\alpha\le \alpha_0$ for some $\alpha_0>0$.  
Then the rescaled critical points   
\[
v_\alpha:=\frac{u_\alpha}{\alpha} 
\]
are bounded in $BV(\Omega)$ for $\alpha\in (0,\alpha_1]$ with some 
$\alpha_1\in(0,\alpha_0]$. Moreover, 

\[
 \text{$(\hat c_\alpha)_\alpha$ converges as $\alpha\to 0$ \; if and only if \;   
 $(\lambda_\alpha)_\alpha$ converges as $\alpha\to 0$\,}
\]
and, in that case, we have
\[ \lim_{\alpha\to 0}\lambda_\alpha=\lim_{\alpha\to 0} \hat c_\alpha\,. \]
\end{prop}

\begin{proof} 
We use \eqref{e:Eper_le_||p}, \eqref{e:ppbound}, and 
$\|u_\alpha\|_1=\alpha$ to estimate 
\begin{align*}
\hat c_\alpha &= \tfrac{1}{\alpha}\big(\cE_{TV}(u_\alpha) +
\cE_{\Per}(u_\alpha)\big)\\ 
 & \ge \cE_{TV}(v_\alpha) - \tfrac{1}{\alpha}C_{\Per}\|\alpha v_\alpha\|_p^p\\
& = \cE_{TV}(v_\alpha) - \alpha^{p-1}C_{\Per}\|v_\alpha\|_p^p\displaybreak[2]\\
&\ge \cE_{TV}(v_\alpha)\! - \alpha^{p-1}C_{\Per} \big(
C_{\textup{BV}}^{(p-1)n}\,\|v_\alpha\|_1^{n-(n-1)p}\,\cE_{TV}(v_\alpha)\! +
C_{\textup{BV}}^{(p-1)n}\,\|v_\alpha\|_1^{n-(n-1)p}\big)\displaybreak[2]\\ 
&= \big(1- \alpha^{p-1}C_{\Per} \,C_{\textup{BV}}^{(p-1)n} \big)
\cE_{TV}(v_\alpha) - \alpha^{p-1}C_{\Per}\,  C_{\textup{BV}}^{(p-1)n}\,. 
\end{align*}
For $\alpha\in(0,\alpha_0]$ sufficiently small, say $\alpha\le\alpha_1$, 
we obtain 
\begin{align}
\cE_{TV}(v_\alpha)\le \frac{\hat c_\alpha + \alpha^{p-1}C_{\Per}\,
  C_{\textup{BV}}^{(p-1)n}}{1-
  \alpha^{p-1}C_{\Per}\,C_{\textup{BV}}^{(p-1)n}}
\le \frac{\hat c_\alpha + \alpha_1^{p-1}C_{\Per}\,
  C_{\textup{BV}}^{(p-1)n}}{1-
  \alpha_1^{p-1}C_{\Per}\,C_{\textup{BV}}^{(p-1)n}}\,.
\label{e:cE(va)}
\end{align}
Hence $\cE_{TV}(v_\alpha)$ is bounded and, since $\cE_{TV}$ is an equivalent
norm on $BV(\Omega)$, the first assertion follows.

Analogously to \eqref{e:cE(va)} we get
\begin{align}
\cE_{TV}(v_\alpha)\ge \frac{\hat c_\alpha - \alpha^{p-1}C_{\Per}
  C_{\textup{BV}}^{(p-1)n}}{1+ \alpha^{p-1}C_{\Per}\,C_{\textup{BV}}^{(p-1)n}
}\,. 
\label{e:cE(va)2} 
\end{align}
By \eqref{e:results5} there is some $u_\alpha^*\in \bd \cE_{\Per}(u_\alpha)$ 
with
\[
\lambda_\alpha =\frac{\cE_{TV}(u_\alpha) +
  \<u_\alpha^*,u_\alpha\>_{L^{p'}\!,L^p}}{\alpha} \,.
\]
Using \eqref{e:u*leu_p^p}, \eqref{e:ppbound}, and \eqref{e:cE(va)} we can
derive that 
\begin{align*}
\lambda_\alpha 
&=\cE_{TV}(\tfrac{u_\alpha}{\alpha}) +
\alpha^{-1}\<u_\alpha^*,u_\alpha\>_{L^{p'}\!,L^p}\\ 
&\le \cE_{TV}(\tfrac{u_\alpha}{\alpha}) + \alpha^{-1}p\,C_{\Per}\|u_\alpha\|_p^p\\
&=  \cE_{TV}(v_\alpha) + \alpha^{p-1}p\,C_{\Per}\|v_\alpha\|_p^p\\
&\le  \Big(1 +
\alpha^{p-1}p\,C_{\Per}\,C_{\textup{BV}}^{(p-1)n}\Big)\cE_{TV}(v_\alpha) +
\alpha^{p-1}p\,C_{\Per}\,C_{\textup{BV}}^{(p-1)n}\\ 
&\le  \Big(1 + \alpha^{p-1}p\,C_{\Per} \; C_{\textup{BV}}^{(p-1)n}\Big)
\frac{\hat c_\alpha + \alpha^{p-1}C_{\Per}\,C_{\textup{BV}}^{(p-1)n}}{1-
  \alpha^{p-1}C_{\Per}\,C_{\textup{BV}}^{(p-1)n}} +
\alpha^{p-1}p\,C_{\Per}\,C_{\textup{BV}}^{(p-1)n}\,. \\
\end{align*}
Consequently,
\begin{align}
 \limsup_{\alpha\to0}\lambda_\alpha\le \liminf_{\alpha\to 0} \hat
 c_\alpha\,.\label{e:ls<li} 
\end{align}
With \eqref{e:cE(va)2} we similarly obtain the opposite direction
\begin{align*}
\lambda_\alpha &= \cE_{TV}(\tfrac{u_\alpha}{\alpha}) +
\alpha^{-1}\<u_\alpha^*,u_\alpha\>_{L^{p'}\!,L^p}\\ 
&\ge \cE_{TV}(\tfrac{u_\alpha}{\alpha}) - \alpha^{-1}p\,C_{\Per}\|u_\alpha\|_p^p\\
&\ge  \Big(1 - \alpha^{p-1}p\,C_{\Per} \; C_{\textup{BV}}^{(p-1)n}\Big) 
\frac{\hat c_\alpha - \alpha^{p-1}C_{\Per} C_{\textup{BV}}^{(p-1)n}}{1+
  \alpha^{p-1}C_{\Per} C_{\textup{BV}}^{(p-1)n}} - \alpha^{p-1}p\,C_{\Per}
C_{\textup{BV}}^{(p-1)n}\\ 
\end{align*}
and, thus,
\begin{align}
 \limsup_{\alpha\to 0} \hat c_\alpha \le \liminf_{\alpha\to 0}\lambda_\alpha
 \,.\label{e:li>ls} 
\end{align}
Now the assertion follows from \eqref{e:ls<li} and \eqref{e:li>ls}.
\end{proof}

\medskip

We continue with the proof of Theorem \ref{t:EperBifurc} by applying
Proposition \ref{p:20} to $(u_{k,\alpha})_\alpha$ and 
$(\lambda_{k,\alpha})_\alpha$. Using \eqref{e:proofs16} we conclude that
\[ 
\lim_{\alpha\to 0}\lambda_{k,\alpha}=\lim_{\alpha\to 0} \hat
c_{k,\alpha}=\lambda_{k,1}
\]
which completes the proof.
\end{proof}

\medskip

\begin{proof}[Proof of Theorem \ref{t:exist:CP1perLap}]
We still have to show that, for fixed $k\in\N$, there is some $\alpha_1>0$
such that the family $v_{k,\alpha}=\frac{u_{k,\alpha}}{\|u_{k,\alpha}\|_1}$
is bounded in $BV(\Omega)$ for $\alpha\in(0,\alpha_1]$. 
But, by \eqref{e:proofs16},
this is a direct consequence of Proposition \ref{p:20} applied to
$(u_{k,\alpha})_\alpha$ and $(\lambda_{k,\alpha})_\alpha$ .
\end{proof}

\subsection{Proofs for perturbations of the constraint}
\label{ss:proofPertconstr}

\begin{proof}[Proof of Theorem~\ref{t:ELGGper}]
It is not difficult to see that we can apply
\cite[Cor.~3.7]{degiovanni-s:98} with 
\[ f_0=\cE_{TV}\,,\; f_1=0\,,\; g_0=-1\,,\; g_1=\cG_1 + \cG_{\Per} -\beta\,.\]
As in the proof of Theorem \ref{t:ELG}, 
the (epi)-con\-di\-tion follows from \cite[Thm.~3.4]{degiovanni-s:98}) with
$u_{-}=0$ and $u_{+}=2u$ by the preceding lemma. 
\end{proof}

\medskip

Let us now prepare the proof of Theorem~\ref{t:CP_Gpert}. 

\begin{lem}\label{l:propG1Per}
Let $\cGP$ be locally Lipschitz continuous such that \textup{(G5)} is satisfied
and let $u\in L^p(\Omega)\setminus\{0\}$. Then the function
\begin{align*}
t\mapsto \cG_1(tu) +\cG_{\Per}(tu)
\end{align*}
is strictly increasing on $[0,\infty)$ and we have 
\begin{align} \label{e:proofs6}
(\cG_1+\cG_{\Per})^0(u;\, -u)<0\,.
\end{align}
\end{lem}
\begin{proof} Let $\cG:=\cG_1 + \cG_{\Per}$ and let $0\le t_1<t_2$. 
By Lebourg's Theorem (cf.~\cite[Thm.~2.3.7]{clarke:87}) 
there is $\theta\in (0,1)$ and $w^*\in \partial \cG\big((\theta t_1 +
(1-\theta)t_2)u\big)$ such that 
\begin{align}
\cG(t_2u)-\cG(t_1u)= \<w^*,(t_2-t_1)u\>=(t_2-t_1)\int_{\Omega}w^*(x) u(x)\d
x\,.\label{e:cGstrMon} 
\end{align}
By the sum rule for generalized gradients
(cf.~\cite[Prop.~2.3.3]{clarke:87}) we find 
$s\in \partial\cG_1(u)$ and
$u^*\in \cG_{\Per}(u)$ with $w^*=s + u^*$ where $s(x)\in \Sgn(u(x))$ for
a.e. $x\in \Omega$ (cf.~\cite{kawohl-s:07}). 
Whence we have for almost every $x\in\Omega$ with $u(x)\neq 0$ that 
\[
w^*(x)u(x) =s(x) u(x) + u^*(x) u(x) = |u(x)| + u^*(x) u(x) > 0
\]
by \eqref{e:u^*u>-|u|2}. 
With \eqref{e:cGstrMon} we obtain the first assertion that 
$t\mapsto \cG(tu)$ is strictly increasing.

Using \cite[Prop.~2.1.2]{clarke:87} and
\eqref{e:u^*u>-|u|2} we get for the 
generalized directional derivative 
\begin{align*}
\cG^0(u;\, -u)&=\max_{w^*\in \partial \cG(u)}\<w^*,-u\>\\
&\le \max_{s\in \partial \cG_1(u)} \<s,-u\> +
\max_{u^*\in\partial\cG_{\Per}(u)} \<u^*,-u\>\\ 
&=-\int_\Omega |u(x)| \d x + \max_{u^*\in\partial \cG_{\Per}(u)}-\int_\Omega
u^*(x) u(x) \d x\\ 
&=-\|u\|_1 - \min_{u^*\in\partial \cG_{\Per}(u)}\int_\Omega u^*(x) u(x) \d x\\
&<-\|u\|_1 - \int_\Omega -|u(x)| \d x\\
&= -\|u\|_1 +\|u\|_1 =0\,
\end{align*}
and the proof is complete.
\end{proof}

\smallskip

\begin{lem}\label{l:t_u}
We assume that the assumptions of Theo\-rem~\ref{t:CP_Gpert} are satisfied 
and that $u\in L^p(\Omega)\setminus\{0\}$ is given. 
Then there exists a unique $t_u>0$ such that
 \[\cG_1(t_uu)+\cG_{\Per}(t_uu)=\beta\,.\]
The mapping $u\mapsto t_u$ is continuous and even on
$L^p(\Omega)\setminus\{0\}$. Moreover
\[
\Phi_\beta : \{u\in L^p(\Omega)\setsep \cG_1(u)+\cG_{\Per}(u)=\beta\}\to
\{u\in L^p(\Omega)\setsep \|u\|_1=1\}
\] 
given by
\[\Phi_\beta(u)=\frac{u}{\|u\|_1}    \]
is an odd homeomorphism with odd inverse $\Psi_\beta$ given by
\[
\Psi_\beta(u) =  t_u u \qquad\text{for $u\in L^p(\Omega)$ with $\|u\|_1=1$}\,.
\]
\end{lem}
\begin{proof}
By Lemma \ref{l:propG1Per} the mapping 
\[ [0,\infty)\ni t\mapsto \|tu\|_1 + \cG_{\Per}(tu)\]
is strictly increasing and, obviously, it is continuous. 
From assumption \eqref{e:G_1est} we infer 
\[
\|0u\|_1 + \cG_{\Per}(0u)=0 \quad\text{and}\quad
\liminf_{t\to\infty} \|tu\|_1 + \cG_{\Per}(tu)\ge \liminf_{t\to \infty}t\|u\|_1=\infty\,.
\]
Thus $t_u$ exists by the intermediate value theorem and is uniquely determined
by strict monotonicity. Since $\cG_1$ and $\cGP$ are even, also $u\mapsto t_u$
is even.

Let now $u_j\to u\ne 0$ in $L^p(\Omega$ and, thus, also in $L^1(\Omega)$. With
$t_j:=t_{u_j}$ and $\cGP(u_j)\ge 0$ by \eqref{e:G_1est} we have 
$\beta\ge t_j\|u_j\|_1$. Therefore $(t_j)_j$ must be bounded and,
at least for a subsequence (denoted the same way), we get $t_j\to: t_0\ge 0$.  
By continuity  
\[
\beta=\lim_{j\to\infty} \cG_1(t_j u_j)+\cG_{\Per}(t_j u_j) = \|t_0u\|_1 +
\cG_{\Per}(t_0u)\,. 
\]
Uniqueness of $t_u$ implies $t_0=t_u$ and, thus, continuity
of $u\mapsto t_u$. The properties of $\Phi_\beta$ and $\Psi_\beta$ are 
a simple consequence of the properties of $t_u$.
\end{proof}

\medskip

\begin{proof}[Proof of Theorem~\ref{t:CP_Gpert}]
We will apply Theo\-rem~\ref{t:exCP} to 
$\cF=\cE_{TV} + I_{\{\cG_1 + \cG_{\Per}=\beta\}}$. 
Properties (F1) and (F2) are immediate. 
Since $\cE_{TV}$ is an equivalent norm on $BV(\Omega)$, the sublevel sets
$\{\cF\le c\}$ are obviously bounded in $BV(\Omega)$ 
and, by the compact embedding $BV(\Omega)\hookrightarrow L^p(\Omega)$,
they are compact in $L^p(\Omega)$.
Clearly, any (PS)-sequence for the level
$c\in \R$ is eventually contained in $\{\cF\le c+ 1\}$ and, therefore,
compactness of all sublevel sets implies the (PS)-con\-di\-tion. 
The (epi)-con\-di\-tion follows from \eqref{e:proofs6} 
and \cite[Thm.~3.4]{degiovanni-s:98} applied with 
\[ g_0=-1\,,\; g_1=\cG_1+\cG_{\Per}-\beta\,,\;
u_{-}=0\,,\; u_{+}=2u\,,\;  C=BV(\Omega)        \,.
\] 
Using $\Phi:\S^{k-1}\to\{\|\cdot\|_1=1\}\subseteq L^p(\Omega)$
from the proof of Theo\-rem~\ref{t:exist:CP1perLap} with $\alpha=1$, the
mapping 
$\Psi_\beta\circ\Phi:\S^{k-1}\to \{\cG_1+\cGP=\beta\}\subseteq L^p(\Omega)$
verifies assumption (F5). 
Now Theo\-rem~\ref{t:exCP} implies the stated
existence of a sequence of eigensolutions of 
\eqref{e:VP_Gpert1},\,\eqref{e:VP_Gpert2} and the unboundedness of 
the sequence of critical values $(c_{k,\beta})_k$. Clearly we can apply
Theorem \ref{t:ELGGper} for the assertion concerning the Euler-Lagrange
equation. 

It remains to show that, for fixed $k\in\N$, the rescaled family 
$v_{k,\beta}=\frac{u_{k,\beta}}{\|u_{k,\beta}\|_1}$
is bounded in $BV(\Omega)$ for $\beta>0$ bounded. We postpone this part of the
proof to the end of this section.  
\end{proof}

\medskip

\begin{proof}[Proof of Theorem~\ref{t:bifb}]
With condition \eqref{e:G_1est} and Corollary~\ref{cor:||pp}, 
we have for any $u\in
L^p(\Omega)\cap BV(\Omega)$ with $\cG_1(u)+\cG_{\Per}(u)=\beta$ that
\begin{align}
\|u\|_1\le \beta &= \|u\|_1 + \cG_{\Per}(u)\label{e:proofs11a}\\
&\le \|u\|_1 + C_{\Per}\|u\|_p^p \label{e:proofs11}\\
&\le \|u\|_1 + C_{\Per}C_{\textup{BV}}^{(p-1)n}\,\|u\|_1^{n-(n-1)p}\cE_{TV}(u)^{(p-1)n}
\notag\displaybreak[2]\\  
&= \|u\|_1\Big(1 + C_{\Per}C_{\textup{BV}}^{(p-1)n} \|u\|_1^{p-1} 
\cE_{TV}\Big(\frac{u}{\|u\|_1}\Big)^{(p-1)n}\,\Big) \label{e:proofs8}
\displaybreak[2]\\
&\le \|u\|_1\Big(1 + C_{\Per}C_{\textup{BV}}^{(p-1)n} \beta^{p-1} 
\cE_{TV}\Big(\frac{u}{\|u\|_1}\Big)^{(p-1)n}\,\Big)\,.\label{e:proofs7}
\end{align}

Consider $\Phi_\beta$ from Lemma \ref{l:t_u} and 
$\cS^1_k$ according to \eqref{e:cSkalpha} with $\alpha=1$. 
Since the genus remains unchanged under homeomorphisms, we have 
\begin{align}
S\in \cS_k^\beta\qquad\text{if and only if}\qquad
\Phi_\beta(S)\in\cS^1_k\,.\label{e:proofs7A}
\end{align}
Therefore
\begin{align}
\hat c_{k,\beta} := \frac{c_{k,\beta}}{\beta}
&= \inf_{S\in\cS_k^\beta}\sup_{u\in S}\,\frac{1}{\beta}\,\cE_{TV}(u)\notag\\
&=\inf_{S\in\cS_k^\beta}\sup_{u\in
  S}\,\frac{\|u\|_1}{\beta}\,\cE_{TV}\Big(\frac{u}{\|u\|_1}\Big)
  \notag\displaybreak[2]\\  
&\le \inf_{S\in\cS_k^\beta}\sup_{u\in
  S}\,\frac{\beta}{\beta}\,\cE_{TV}\Big(\frac{u}{\|u\|_1}\Big)\notag\\ 
&=\inf_{S\in\cS^1_k}\sup_{v\in S}\,\cE_{TV}(v) = \lambda_{k,1}\,.\label{e:Ckb_le_lk} 
 \end{align}

For some reverse inequality we choose   
$S_\beta\in \cS_k^\beta$ with
\[
\hat c_{k,\beta}=\frac{c_{k,\beta}}{\beta}=\sup_{u\in
  S_\beta}\frac{\cE_{TV}(u)}{\beta}\,
\]
for any $\beta>0$ according to Theorem~\ref{t:exCP}.  
By \eqref{e:Ckb_le_lk} and \eqref{e:proofs8} we have for $u\in S_\beta$
\begin{align}
 \lambda_{k,1}\ge \hat c_{k,\beta} &= \sup_{u\in
   S_\beta}\frac{\cE_{TV}(u)}{\beta} \notag\\
&\ge \sup_{u\in S_\beta} \frac{\cE_{TV}(u)}{\|u\|_1\big(1 +
  C_{\Per}\,C_{BV}^{(p-1)n}\|u\|_1^{p-1}\cE_{TV}\big(\frac{u}{\|u\|_1}\big)^{(p-1)n}
  \big)}\:. \label{e:proofs9}
\end{align}
Consequently,
\[\lambda_{k,1}\left(1 + C_{\Per}C_{BV}^{(p-1)n}
  \beta^{p-1}\cE_{TV}\Big(\frac{u}{\|u\|_1}\Big)^{(p-1)n}\right)\ge
\cE_{TV}\Big(\frac{u}{\|u\|_1}\Big) \quad\text{for }
u\in \bigcup_{\beta>0} S_\beta\,.
\] 

The above inequality is of linear growth in $\cE\big(\frac{u}{\|u\|_1}\big)$
on the right had side and, by $p<1+\frac{1}{n}$,   
of sublinear growth on the left hand side. Thus, for any $\beta_0>0$ there is
some $\ti C>0$ such that 
\begin{align}
 \cE_{TV}\Big(\frac{u}{\|u\|_1}\Big)\le\tilde C \quad\text{for }
  u\in S_\beta\,, \; 0<\beta\le\beta_0\,.   \label{e:Ctilde}
\end{align}
By \eqref{e:proofs7}, \eqref{e:proofs7A}, and the definition of 
$\lambda_{k,1}$ in \eqref{e:la1k1Lap},
we now find some $C>0$ with
\begin{align*}
 \hat c_{k,\beta}= \sup_{u\in S_\beta}\frac{\cE_{TV}(u)}{\beta} 
&\ge 
\frac{1}{1+ C\beta^{p-1}} 
\sup_{u\in S_\beta}\: \cE_{TV}\Big(\frac{u}{\|u\|_1}\Big) \\ 
 &\ge \frac{\lambda_{k,1}}{1+ C\beta^{p-1}} \quad\text{for } \beta\le\beta_0\,.
\end{align*}
Using \eqref{e:Ckb_le_lk} we readily derive the first assertion that
\begin{align}\label{e:proofs12}
 \lim_{\beta\to 0}\frac{c_{k,\beta}}{\beta} =
 \lim_{\beta\to 0} \hat c_{k,\beta} = \lambda_{k,1}\,.
\end{align}

For the other limit in \eqref{e:ckb/b=lb=l}
we first recall that 
$c_{k,\beta}=\cE_{TV}(u_{k,\beta})$. 
Testing the Eu\-ler-La\-grange Equation \eqref{e:ELGGper} with
$u_{k,\beta}$, we obtain
\begin{align}\label{e:proofs13}
\hat c_{k,\beta}=\frac{\cE_{TV}(u_{k,\beta})}{\beta}=\lambda_{k,\beta}\,
\frac{\|u_{k,\beta}\|_1+\<u^*_{k,\beta},u_{k,\beta}\>}{\beta}\, 
\end{align}
for some $u^*_{k,\beta}\in\partial\cG_{\Per}(u_{k,\beta})$.
Thus the remaining result follows if we show that
\begin{align} \label{e:proofs14}
\lim_{\beta\to 0}\frac{\|u_{k,\beta}\|_1 + \<u^*_{k,\beta},u_{k,\beta}\>}{\beta}=1\,.
\end{align}
By the continuous 
embedding $BV(\Omega)\hookrightarrow L^p(\Omega)$ and
since $\cE_{TV}$ is an equivalent norm on $BV(\Omega)$,
there is some $C>0$ such that
\[  
  \|v\|_p\le C \cE_{TV}(v) \quad\text{for } v\in BV(\Omega)\,.
\]
Using \eqref{e:u^*le_u^p-1} we get for some possibly larger $C>0$
\[
  |\<u^*_{k,\beta},u_{k,\beta}\>|\le p\,C_{\Per}\|u_{k,\beta}\|_p^p 
  \le C\,\cE_{TV}(u_{k,\beta})^p\,.
\]
Thus, with \eqref{e:proofs11a} and \eqref{e:proofs12},
\[
\limsup_{\beta \to 0}\frac{\|u_{k,\beta}\|_1 + 
\<u^*_{k,\beta},u_{k,\beta}\>}{\beta}\le\limsup_{\beta\to 0}\frac{\beta + 
C\, c_{k,\beta}^{\,p}}{\beta}=1\,.
\]   
Since $\|u_{k,\beta}\|_1\ge \beta - C_{\Per}\|u_{k,\beta}\|_p^p$ 
by \eqref{e:proofs11}, we find some $\hat C>0$ with  
\begin{align*}
\liminf_{\beta\to 0}\frac{\|u_{k,\beta}\|_1 + \< u^*_{k,\beta},u_{k,\beta} \> }{\beta}
&\ge \liminf_{\beta\to0}\frac{\beta - \hat C c_{k,\beta}^{\,p}}{\beta} = 1\,.
\notag
\end{align*}
But this verifies \eqref{e:proofs14} and the proof is complete. 
\end{proof}

\medskip

\begin{proof}[Proof of Theorem~\ref{t:CP_Gpert}, second part]
Let $k\in\N$ and any $\beta_0>0$ be fixed. For (arbitrary) critical points 
$u_{k,\beta}$ of \eqref{e:VP_Gpert1}, \eqref{e:VP_Gpert2} with critical
value $c_k$ it remains to show that
the family $v_{k,\beta}=\frac{u_{k,\beta}}{\|u_{k,\beta}\|_1}$
is bounded in $BV(\Omega)$ for $0<\beta<\beta_0$ .
By \eqref{e:Ckb_le_lk} and
\eqref{e:proofs7} we obtain 
\[
\lambda_{k,1}\ge
\frac{c_{k,\beta}}{\beta}=\frac{\cE_{TV}(u_{k,\beta})}{\beta}\ge
\frac{\cE_{TV}(u_{k,\beta})}{\|u_{k,\beta}\|_1
\Big(1 + C_{\Per}C_{\textup{BV}}^{(p-1)n} \beta^{p-1}  
\cE_{TV}\Big(\frac{u_{k,\beta}}{\|u_{k,\beta}\|_1}\Big)^{(p-1)n}\,\Big)} \:.
\]
Consequently, for $0<\beta<\beta_0$, 
\[
\lambda_{k,1}\big(1 + C_{\Per}C_{\textup{BV}}^{(p-1)n} \beta_0^{p-1} 
\cE_{TV}(v_{k,\beta})^{(p-1)n}\,\big) \ge
\cE_{TV}( v_{k,\beta})\,. 
\] 
Analogously to the arguments giving \eqref{e:Ctilde}, 
we use the sublinear and linear growth in $\cE_{TV}( v_{k,\beta})$ 
to derive a uniform bound on $\cE_{TV}( v_{k,\beta})$ for
$0<\beta<\beta_0$. 
Since $\cE_{TV}$ is an equivalent norm on $BV(\Omega)$, the assertion follows.
\end{proof}


\begin{thebibliography}{10}

\bibitem{ambrosio-t:04}
L.~Ambrosio, P.~Tilli.
\newblock {\em Topics on Analysis in Metric Spaces}.
\newblock Oxford University Press, Oxford, 2004.

\bibitem{cioranescu:90}
I.~Cioranescu.
\newblock {\em Geometry of Banach Spaces, Duality Mappings and Nonlinear
  Problems}.
\newblock Kluwer Academic Publishers, Dordrecht, 1990.

\bibitem{clarke:87}
F.~Clarke.
\newblock {\em Optimization and Nonsmooth Analysis}.
\newblock Canadian Mathematical Society Series of Monographs and Advanced Texts,  John Wiley \& Sons, New York, 1987.

\bibitem{corvellec:97}
J.-N. Corvellec.
\newblock A general approach to the min-max principle.
\newblock {\em Z. Anal. Anwend.}
  16  (1997) 405--433.

\bibitem{corvellec-dm:93}
J.-N. Corvellec, M.~Degiovanni, M.~Marzocchi.
\newblock Deformation properties for continuous functionals and critical point
  theory.
\newblock {\em Topol. Methods Nonlinear Anal.}  1  (1993) 151--171.

\bibitem{degiovanni-m:09}
M.~Degiovanni, P.~Magrone.
\newblock Linking solutions for quasilinear equations at critical growth
  involving the "$1$-Laplace" operator.
\newblock {\em Calc. Var. Partial Differential Equations}  36
  (2009) 591--609.

\bibitem{degiovanni-m:94}
M.~Degiovanni, M.~Marzocchi.
\newblock A critical point theory for nonsmooth functionals.
\newblock {\em Ann. Mat. Pura Appl. (IV)}  167  (1994)
  73--100.

\bibitem{degiovanni-s:98}
M.~Degiovanni, F.~Schuricht.
\newblock Multiplicity results for free and constrained nonlinear elastic rods
  based on nonsmooth critical point theory.
\newblock {\em Math. Ana.}  331  (1998) 675--728.

\bibitem{delpino-m:91}
M.~A. del Pino, R.~F. Man{\'a}sevich.
\newblock Global bifurcation of the $p$-Laplacian.
\newblock {\em J. Differ. Equ.}  92  (1991) 226--251.

\bibitem{evans-g:92}
L.~C. Evans, R.~F. Gariepy.
\newblock {\em Measure Theory and Fine Properties of Functions (Studies in Advanced Mathematics)}.
\newblock CRC-Press, Boca Raton, 1992.

\bibitem{fadell:80}
E.~Fadell.
\newblock The relationship between Ljus\-ter\-nik-Schni\-rel\-man
 category and the
  concept of genus.
\newblock {\em Pacific J. M.}  89  (1980)  33--42.

\bibitem{azorero-a:87}
J.~P. Garcia~Azorero, I.~Peral~Alonso.
\newblock Existence and nonuniqueness for the $p$-Laplacian.
\newblock {\em Comm. Partial Differential Equations}  12  (1987)
  1389--1430.

\bibitem{kawohl-s:07}
B.~Kawohl, F.~Schuricht.
\newblock Dirichlet problems for the $1$-Laplace operator, including the
  eigenvalue problem.
\newblock {\em Commun. Contemp. Math.}  9
  (2007) 513--543.

\bibitem{littig-s:13}
S.~Littig, F.~Schuricht.
\newblock Convergence of the eigenvalues of the $p$-Laplace operator as $p$ goes to $1$.
\newblock {\em Calc. Var. Partial Differential Equations}. 
49 (2014) 707--727.

\bibitem{milbers-s:10}
Z.~Milbers, F.~Schuricht.
\newblock Existence of a sequence of eigensolutions for the $1$-Laplace operator.
\newblock {\em J. Lond. Math. Soc.}  82  (2010) 74--88.

\bibitem{milbers-s:10a}
Z.~Milbers, F.~Schuricht.
\newblock Some special aspects related to the $1$-Laplace operator.
\newblock {\em Adv. Calc. Var.}  4  (2010) 101--126.

\bibitem{milbers-s:12}
Z.~Milbers, F.~Schuricht.
\newblock Necessary condition for eigensolutions of the $1$-Laplace
operator by means of inner variations.
\newblock {\em Math. Ann.} 356 (2013) 147-177. 

\bibitem{parini:11}
E. Parini.
\newblock Continuity of the variational eigenvalues of the $p$-Laplacian with respect to $p$.
\newblock {\em Bull. Aust. Math. Soc.}  83  (2011)
  376--381.

\bibitem{peral:97}
I. Peral.
\newblock Multiplicity of solutions for the $p$-Laplacian. 
\newblock Intern. Center for Theoretical Physics Trieste, 
Second School of Nonlinear Functional Analysis and Applications to
Differential Equations, 21 April - 9 May 1997. 

\bibitem{rabinowitz:73}
P.~Rabinowitz.
\newblock Some aspects of nonlinear eigenvalue problems.
\newblock {\em Rocky Mountain J. Math.}  3  (1973) 161--202.

\bibitem{zeidler:84}
E.~Zeidler.
\newblock {\em Nonlinear Functional Analysis and its Applications: Part 3:
  Variational Methods and Optimization}.
\newblock Springer, Berlin, 1984.

\end{thebibliography}
\end{document}